\documentclass[dvips]{article}
\usepackage[utf8]{inputenc}
\usepackage[english]{babel}
\usepackage{amssymb}
\usepackage{amsmath}
\usepackage{amsthm}
\usepackage{dsfont}
\usepackage{mathrsfs}
\usepackage[pdftex]{graphicx}
\usepackage{caption}
\usepackage{enumerate}

\DeclareGraphicsExtensions{.eps}

\newtheorem{thm}{Theorem}[section]
\newtheorem{prop}[thm]{Proposition}
\newtheorem{lem}[thm]{Lemma}
\newtheorem{cor}[thm]{Corollary}

\theoremstyle{definition}
\newtheorem{defn}[thm]{Definition}

\theoremstyle{remark}
\newtheorem{rmk}[thm]{Remark}

\newcommand{\ind}{\mathbf{ind}}
\newcommand{\Sv}{\Sigma_\varphi}
\newcommand{\Svp}{\Sigma_{\varphi_{per}}}
\newcommand{\Sm}{\Sigma_\varphi^-}
\newcommand{\Sp}{\Sigma_\varphi^+}
\newcommand{\gm}{\gamma_{\varphi_-}}
\newcommand{\gp}{\gamma_{\varphi_+}}
\newcommand{\gmk}{\gamma_{\varphi_-^k}}
\newcommand{\gpk}{\gamma_{\varphi_+^k}}

\date{\today}
\author{\textsc{Jullian} Yann}
\title{Explicit computation of the index of a positive outer automorphism of the free group}

\addcontentsline{toc}{section}{Introduction}

\begin{document}
\pagenumbering{arabic}
%\pagenumbering{roman}
\begin{center}
\vbox{
\vspace*{4.7cm}
{\LARGE Explicit computation of the index of a positive outer automorphism of the free group}

\vspace*{0.5cm}

{\large Yann {\scshape Jullian}}\\
\vspace{2mm}
\small{\textit{Université Paul Cézanne, LATP\\
Avenue de l'Escadrille Normandie-Niémen, Case A\\
13397 Marseille cedex 20, France\\
yann.jullian@gmail.com}}

\vspace*{0.8cm}

\begin{abstract}
We give an algorithm for finding the index of a positive outer automorphism of the free group,
and prove the algorithm exits in a finite time.
\end{abstract}

}
\end{center}

\setcounter{tocdepth}{2} % Fous 2 si tu veux les subsection.
\tableofcontents{}
%\pagenumbering{arabic}

%Dire $\mathds{N}$ et $\mathds{N}^*$ ?\\

\section*{Introduction}

We consider an automorphism $\varphi$ of the free group $F_N$ ($N\ge 2$). The Scott conjecture, proven by Bestvina-Handel (\cite{BH})
states that the fixed subgroup $Fix(\varphi) = \{u\in F_N;~\varphi(u)=u\}$ of $\varphi$ has rank at most $N$.
Another proof of this result was given several years later in \cite{GLL}, and was improved soon after in \cite{GJLL}
by adding a key ingredient to the formula. The automorphism $\varphi$ induces a homeomorphism $\partial \varphi$ on
the Gromov boundary $\partial F_N$ of $F_N$. Define $a(\varphi) = \#(Att(\partial \varphi)/Fix(\varphi))$ as the
number of equivalence classes of points of $\partial F_N$ that are attracting (in the topological sense) for $\partial \varphi$.
In the article \cite{GJLL}, the index of $\varphi$ is defined as $\ind(\varphi) = rk(Fix(\varphi)) +\frac{1}{2}a(\varphi)-1$
and it is proven that $\ind(\varphi)\le N-1$. Further investigations are made, and the index of an outer automorphism is defined.
The outer class $\Phi$ of the automorphism $\varphi$ is the set $\{i_w\circ \varphi;~w\in F_N\}$ where $i_w$ is a
conjugacy; for any $u\in F_N$, $i_w(u) = w^{-1}uw$. The index of $\Phi$ is defined by
$$\ind(\Phi) = \sum\limits_{[\psi]}\max(0, \ind(\psi))$$
where the sum is taken over all isogredience classes of $\Phi$; two automorphisms $\psi$ and $\chi$ of $\Phi$ are isogredient if there is $v\in F_N$
such that $\chi = i_v\circ\psi\circ i_{v^{-1}}$. It is again stated in \cite{GJLL} that $\ind(\Phi)\le N-1$.

In the present article, we wish to give an algorithm able to compute explicitely (for a certain class of automorphisms) the
FO-index (for Full Outer index) of $\varphi$
defined by $\ind(\Phi_\infty) = \max\limits_{k\in \mathds{N}^*}(\ind(\Phi^k))$ where $\Phi^k$ is the outer class of $\varphi^k$.
In particular, the algorithm will indicate the relevant isogredience classes, give a basis of the fixed subgroups, and give representatives of
the equivalence classes of attracting points.\\

This index is of particular relevance when attempting to give a geometric interpretation of the dynamics of automorphisms of free groups.
It is now standard (\cite{GLL}, \cite{GJLL}, \cite{LL03}) to represent the dynamics of an automorphism $\varphi$
with an $\mathds{R}$-tree (a geodesic and $0$-hyperbolic metric space) $T_\varphi$ on which the free group acts isometrically and
the automorphism acts as a homothety; such a tree is called the invariant tree of $\varphi$.
In fact, the proofs of \cite{GLL} and \cite{GJLL} are mostly of a geometric nature, using this invariant tree.
In \cite{CHL09}, it is stated that the whole dynamics on $T_\varphi$ can be deduced from studying the induced
isometric action of the free group on certain compact subsets of $\overline{T_\varphi}$,
the metric completion of $T_\varphi$. These subsets are called the limit sets of $T_\varphi$,
and each one is associated to a basis of the free group. A detailed analysis
of these limit sets is made in \cite{CH}, and it is proven that the general structure of these sets
does not depend on the basis chosen, but on the FO-indices of $\varphi$ and its inverse.
Depending on these indices, the limit sets can either be a finite union of finite (with respect to the number
of points with degree at least three) trees, a finite union of non finite trees, a cantor set whose
convex hull is a finite tree, or a cantor set whose convex hull is a non finite tree.
In \cite{BK}, an example of cantor set (with an interval as convex hull) is treated, and \cite{Jul} deals
with a non finite tree. In any case, there are still some unanswered questions regarding the dynamics on these limit sets
and their understanding starts with the FO-index.\\

Section \ref{sec:auts} recalls classical definitions of symbolic dynamics and group theory.
We endow $F_N$ with a basis $A_N$ and we see the free group as the set of words with letters in $A_N$ or $A_N^{-1} = \{a^{-1};~a\in A_N\}$.
We assume the automorphism $\varphi$ is positive (for any $a\in A_N$, all the letters of $\varphi(a)$
are in $A_N$) and primitive (there exists an integer $k$ such that for any $a\in A_N$, all the letters of $A_N$
are contained in $\varphi^k(a)$).

In section \ref{sec:attsubcontainsall}, we prove that the FO-index of $\varphi$ can be obtained by only studying a compact
subset of the double boundary $\partial^2 F_N = (\partial F_N\times \partial F_N)\setminus \Delta$ (where $\Delta$ is the diagonal)
of $F_N$. This subset only depends on $\varphi$; it is called the attracting subshift of $\varphi$ and is denoted $\Sv$.
This eventually leads up to new formulas for the FO-index. These formulas depend on one key element; the singularities of
the attracting subshift.

Here, singularities are introduced as combinatorial objects. Namely, they are finite sets of
at least two elements of $\Sv$ that are all fixed points of a common homeomorphism $\partial^2 \psi$
(where $\partial^2 \psi$ is the homeomorphism induced by $\psi$ on $\partial^2 F_N$) with
$\psi = i_w\circ \varphi^k$ for some $w\in F_N$ and $k\ge 1$.
It is important to observe that singularities have a natural geometric interpretation when dealing with an automorphism
coming from a pseudo-Anosov on a surface, as they correspond to the singularities of the stable foliation.

We define the singularity graph, which contains all the information needed to compute the FO-index.
The singularity graph is determined by a set of nodes (the singularities),
a set of finite edges which are joining two singularities, and a set of infinite edges which are
anchored at a single node. The main result of section \ref{sec:attsubcontainsall} is the following.
\begin{thm}
Let $\varphi$ be an $A_N$-positive primitive automorphism and let $\mathcal{G}$
be the singularity graph associated to $\varphi$. Suppose the automorphism $\psi = i_w\circ \varphi^k$
for some $w\in F_N$ and $k\ge 1$ is such that $\partial^2 \psi$ fixes all the points of a singularity $\Omega$,
and define $\mathcal{G}_\Omega$ as the connected component of $\mathcal{G}$ containing $\Omega$. Then we have:
\begin{itemize}
	\item the number of equivalence classes of attracting points of $\partial \psi$ is equal to the number of infinite edges
	of $\mathcal{G}_\Omega$,
	\item the rank of the fixed subgroup of $\psi$ is equal to the rank of the fundamental group of $\mathcal{G}_\Omega$.
\end{itemize}
\end{thm}
The singularity graph also allows to determine a base of the fixed subgroup of $\psi$ and give finite expressions
of representatives of the equivalence classes of attracting points. It should be noted that a similar graph
is defined in \cite{CL}, and although its primary use was to determine the fixed subgroup of an automorphism,
the graph also contains attracting points. In fact, the present article shares several similarites with \cite{CL}.

It all comes down to finding singularities, which is the aim of the study of section \ref{sec:identsing}.
We obtain the following result.
\begin{thm}
There exists an algorithm able to find all the singularities in a finite number of steps.
\end{thm}
We use the properties of a common tool in symbolic dynamics;
the prefix-suffix automaton. We explain in section \ref{subsec:identifying} how this
automaton allows us to identify singularities with relative ease, and an algorithm
identifying all the singularities is detailed in section \ref{subsec:algorithm}.

We end the article by giving two typical examples (section \ref{sec:examples}).

\section{Automorphisms of the free group}\label{sec:auts}

Let $F_N$ be the free group on $N\ge 2$ generators and let $\partial F_N$ be its Gromov boundary. The \textbf{double boundary} $\partial^2F_N$ is defined by
\begin{center}
	$\partial^2F_N = (\partial F_N\times \partial F_N) \setminus \Delta$,
\end{center}
where $\Delta$ is the diagonal.

Let $A_N = \{a_0, \dots, a_{N-1}\}$ be a basis of $F_N$. The set of inverse letters is denoted by $A_N^{-1} = \{a_0^{-1}, \dots, a_{N-1}^{-1}\}$.
Fixing $A_N$ as a basis, we consider $F_N$ to be the set of finite \textbf{reduced} words $v=v_0v_1\dots v_p$ with letters in $(A_N\cup A_N^{-1})$;
reduced means for all $0\le i < p$, we have $v_i^{-1}\ne v_{i+1}$. The length of the word $v$ is $|v| = p+1$.
If $u$ and $u'$ are words of $F_N$, writing $|uu'|$ refers to the length of the reduced word $v$
defined by $v=uu'$. The identity element of $F_N$ is identified with the empty word $\epsilon$ and has length $0$.

We will often need to specify that the concatenation of two elements is cancellation free.
Following from \cite{CL}, for $u\in F_N$ and $u'\in F_N$, we write $u*u'$ if no
cancellations occur between $u$ and $u'$; in other words, if $u_p$
is the last letter of $u$ and $u_0'$ is the first letter of $u'$, then $u_p^{-1}\ne u_0'$.
This notation will be used very frequently throughout this article.\\

With $A_N$ as basis of $F_N$, the set $\partial F_N$ is the set of points $V=(V_i)_{i\in \mathds{N}}$ with letters in $(A_N\cup A_N^{-1})$
and such that $V_i^{-1}\ne V_{i+1}$ for any $i\in \mathds{N}$.

The free group $F_N$ acts continuously on $\partial F_N$ by left translation: if $v=v_0\dots v_p\in F_N$ and $V=V_0V_1\dots \in \partial F_N$,
then $vV = v_0\dots v_{p-i-1}V_{i+1}\dots V_p\dots$, where $V_0\dots V_i = v_p^{-1}\dots v_{p-i}^{-1}$ is the longest common prefix of
$v^{-1}$ and $V$, is in $\partial F_N$. We write again $v*V$ if no cancellation occur between
$v$ and $V$. Obviously, $F_N$ also acts on $\partial^2 F_N$: if $v\in F_N$ and $(U, V)\in \partial^2 F_N$, then $v(U, V) = (vU, vV)\in \partial^2 F_N$

\begin{rmk}
In order to avoid possible confusions, sequences of elements of $F_N$ or $\partial F_N$ will be denoted $(V_{(n)})_n$
rather than simply $(V_n)_n$, and we will keep the notation $V_n$ to refer to the $n$th letter of $V$.
\end{rmk}
\begin{rmk}
We will refer to elements of $F_N$ as words, while elements of $\partial F_N$ and $\partial^2 F_N$
will be called points.
\end{rmk}

A word $u\in F_N$ is a \textbf{prefix} of an element $V\in F_N\cup \partial F_N$ if $V = u*V'$ for some $V'\in F_N\cup \partial F_N$. The word $u$
is a \textbf{suffix} of the word $v\in F_N$ if $v = v'*u$ for some $v'\in F_N$. A prefix or suffix $u$ of a word $v$
is said to be \textbf{strict} if $u\ne v$.
We will say that $V\in F_N\cup \partial F_N$ is \textbf{pure positive} (resp. \textbf{pure negative}) if all of its letters are in $A_N$
(resp. $A_N^{-1}$). The word $V$ is \textbf{pure} if it is either pure positive or pure negative. We assume the empty word $\epsilon$
to be neither pure positive nor pure negative. Lastly,
for a word $v$ of $F_N$, a pair $v_iv_{i+1}$, where $v_i\in A_N$ (resp. $v_i\in A_N^{-1}$)
and $v_{i+1}\in A_N^{-1}$ (resp. $v_{i+1}\in A_N$) is called an \textbf{orientation change}.\\

An automorphism $\varphi$ of $F_N$ is $\boldsymbol{A_N}$\textbf{-positive} if for all $a\in A_N$, $\varphi(a)$ is pure positive.
When working with an $A_N$-positive automorphism, we always consider its representation over the free group endowed with
the basis $A_N$. The automorphism $\varphi$ induces a homeomorphism $\partial\varphi$ on $\partial F_N$, and
a homeomorphism $\partial^2\varphi$ on $\partial^2 F_N$.\\

Throughout this paper, $F_N$ will refer to the free group endowed with the basis $A_N$ and we always assume $N\ge 2$.

	\subsection{The attracting subshift}
	We define the shift map $S$ on $\partial^2 F_N$:
	\begin{center}
		\begin{tabular}{cccc}
			$S~:$ & $\partial^2 F_N$ & $\to$ & $\partial^2 F_N$\\
			& $(X, Y)$ & $\mapsto$ & $(Y_0^{-1}X, Y_0^{-1}Y)$,
		\end{tabular}
	\end{center}
	where $Y_0$ is the first letter of $Y$. For any point $Z\in \partial^2 F_N$, we define
	the $\boldsymbol{S}$-\textbf{orbit} of $Z$ as the set $\{S^n(Z);~n\in \mathds{Z}\}$.

	An $A_N$-positive automorphism is \textbf{primitive} if there is a positive integer $k$ such that all letters of $A_N$
	are letters of $\varphi^k(a)$ for any $a\in A_N$.\\

	Let $\varphi$ be an $A_N$-positive primitive automorphism and
	let $a$ be a letter of $A_N$. The primitivity condition implies that we can find an integer $k$
	such that $\varphi^k(a) = p*a*s$ where $p$ and $s$ are non empty pure positive words of $F_N$.
	Now define
	\begin{center}
		$X = \lim\limits_{n\to +\infty} p^{-1}\varphi^k(p^{-1})\varphi^{2k}(p^{-1})\dots \varphi^{nk}(p^{-1})$,\\
		$Y = \lim\limits_{n\to +\infty} as\varphi^k(s)\varphi^{2k}(s)\dots \varphi^{nk}(s)$.
	\end{center}
	The \textbf{attracting subshift} of $\varphi$ is the closure of the $S$-orbit of the point $(X, Y)$:
	\begin{center}
		$\Sv = \overline{\{S^n(X, Y); n\in \mathds{Z}\}}$.
	\end{center}
	The map $S$ is a homeomorphism on $\Sv$.
	The attracting subshift only depends on $\varphi$ and not on the choice of the letter $a$ or the integer $k$
	(this is stated in different contexts in both \cite{Que} and \cite{BFH} (for example)).
	Note that $\partial^2 \varphi(\Sv)$ is a subset of $\Sv$.
	For any power $k\in \mathds{N}^*$ of $\varphi$, the automorphism $\varphi^k$ is also $A_N$-positive, primitive,
	and its attracting subshift is $\Sigma_{\varphi^k} = \Sv$.

	The projection of $\Sv$ on its first (resp. second) coordinate will be denoted $\Sm$ (resp. $\Sp$).
	The set $\Svp\subset \Sv$ of periodic points of $\partial^2\varphi$ is defined by
	\begin{center}
		$\Svp = \{(X, Y)\in \Sv;~\exists k\in \mathds{N}^*;~\partial^2 \varphi^k(X, Y) = (X, Y)\}$.
	\end{center}
	Finally, we define the \textbf{half lamination} as the set $L(\varphi) = F_N.(\Sm\cup \Sp)$.

	\begin{rmk}
	We will most often keep the notation $U$ (resp. $V$, $(U, V)$) to refer to a point in $\Sm$ (resp. $\Sp$, $\Sv$),
	while generic points of $\partial F_N$ will be denoted $X$ or $Y$.
	\end{rmk}

	\subsection{The index of an automorphism}\label{subsec:index}

	The index of an automorphism $\varphi$ (see \cite{GJLL}) of the free group $F_N$ is defined by:
	$$
		\ind(\varphi) = rk (Fix (\varphi)) + \frac{1}{2}(a(\varphi) - 2),
	$$
	where $Fix (\varphi) = \{u\in F_N;~\varphi(u)=u\}$ and $a(\varphi)$ is the number of equivalence classes of attracting points of $\partial F_N$:
	specifically, $a(\varphi) = \#(Att (\partial \varphi)/Fix (\varphi))$. Here, the word attracting should be understood in the topological sense.

	The set of conjugacy $i_w:u\mapsto w^{-1}uw$ of $F_N$ is denoted $Inn(F_N)$. Recall that the outer class $\Phi$ of $\varphi$ is the set of automorphisms
	$\{i_w\circ \varphi;~i_w\in Inn(F_N)\}$. Two automorphisms $\psi$ and $\chi$ of $\Phi$ are \textbf{isogredient} if there is $v\in F_N$
	such that $\chi = i_v\circ\psi\circ i_{v^{-1}}$. The index of the outer automorphism $\Phi$ (see \cite{GJLL}) is defined by:
	\begin{center}
		$\ind(\Phi) = \sum\limits_{[\psi]} \max (0, \ind(\psi))$,
	\end{center}
	where the sum is taken over all isogredience classes of $\Phi$.\\

	The aim of this paper is to provide an algorithm able to find the index defined below.
	\begin{defn}
	Let $\Phi_\infty$ be the set $\Phi_\infty = \bigcup\limits_{k\in \mathds{N}^*} \Phi^k$, where $\Phi^k$ is the outer class of $\varphi^k$.
	The \textbf{FO-index} (for Full Outer index) of $\varphi$ is defined by
	\begin{center}
		$\ind(\Phi_\infty) = \max\limits_{k\in \mathds{N}^*} (\ind(\Phi^k))$.
	\end{center}
	\end{defn}
	The reader is referred to \cite{GJLL} for a geometric approach of the problem, and for a proof of the following result.
	\begin{thm}[\cite{GJLL}]\label{thm:gjll}
		Let $\varphi$ be an automorphism of $F_N$ and let $\Phi_\infty$ be the set $\{i_w\circ \varphi^k;~i_w\in Inn(F_N),~k\in \mathds{N}^*\}$.
		Then we have:
		\begin{center}
			$\ind(\Phi_\infty) \le N-1$.
		\end{center}
	\end{thm}

	\subsection{Prefix-suffix representation}\label{subsec:psr}
	In this section, we slightly adapt definitions and results given by V. Canterini and A. Siegel in \cite{CanSie}.
	In order to be able to use these concepts freely, we first need to check that an $A_N$-positive primitive automorphism
	can not be $S$-periodic (its attracting subshift is not finite).
	\begin{prop}\label{prop:notSper}
	An $A_N$-positive primitive automorphism $\varphi$ is not $S$-periodic.
	\end{prop}
	\begin{proof}
	Suppose $\varphi$ is $S$-periodic. Any point $W=(U, V)$ of its attracting subshift $\Sv$ is $S$-periodic: there exists
	a non empty pure positive word $v$ of $F_N$ such that $W=vW$. We may assume $V$ is fixed by $\partial \varphi$ (if not, we simply
	take a power of $\varphi$).
	Since $V=vV$ is also a fixed point of $\partial \varphi^{-1}$, we have $V=\varphi^{-1}(v)V$, which gives $\varphi^{-1}(v)=v^h$ and
	$v=\varphi(v^h)$ for some $h\in \mathds{Z}$.

	If $h<0$, then $v^h$ and $\varphi(v^h)$ are pure negative, while $v$ is pure positive.

	If $h=0$, then $v = v^h = \varphi(v^h) = \epsilon$ and we have a contradiction.

	If $h>0$, then $|v^h|\ge |v|$. Since $\varphi$ is primitive, the word $v$ contains all the letters of $A_N$ and
	$|\varphi(v^h)| > |v^h| > |v|$, which is impossible. 
	\end{proof}

	Let $\varphi$ be an $A_N$-positive primitive automorphism and $\Sv$ its attracting subshift.
	The \textbf{prefix-suffix automaton} of $\varphi$ is defined as follows:
	\begin{itemize}
		\item $A_N$ is its set of vertices,
		\item $P = \{(p, a, s)\in F_N\times A_N\times F_N;~\exists b\in A_N;~\varphi(b) = p*a*s\}$ is the set of labels,
		\item there is an edge labeled $(p, a, s)$ from $a$ to $b$ if and only if $\varphi(b) = p*a*s$.
	\end{itemize}
	The set of sequences of labels of infinite walks in this automaton is denoted $D$ and is called the set of admissible developments.
	\begin{prop}
	If $(p_i, a_i, s_i)_{i\ge 0}$ is in $D$, then for all $n\in \mathds{N}$, $\varphi(a_{n+1}) = p_n*a_n*s_n$.
	\end{prop}
	In \cite{CanSie}, the authors define a map
	$\rho_{\varphi}:\Sv\to D$ which gives a representation of the attracting subshift's structure. We will refer
	to $\rho_{\varphi}$ as the \textbf{prefix-suffix development map} associated to $\varphi$. If $W = (U, V)$ is a point of $\Sv$,
	then $\rho_{\varphi}(W) = (p_i, a_i, s_i)_{i\ge 0}$ is called the \textbf{prefix-suffix development} of $W$, and
	it is such that
	\begin{itemize}
		\item if $(s_i)_{i\in \mathds{N}}$ is not eventually constant equal to $\epsilon$,
		then $V = \lim\limits_{n\to +\infty} a_0s_0\varphi(s_1)\dots \varphi^n(s_n)$,
		\item if $(p_i)_{i\in \mathds{N}}$ is not eventually constant equal to $\epsilon$,
		then $U = \lim\limits_{n\to +\infty} p_0^{-1}\varphi(p_1^{-1})\dots \varphi^n(p_n^{-1})$.
	\end{itemize}
	Those developments whose prefix or suffix sequence end up being constant equal to $\epsilon$ are identified in \cite{CanSie},
	and they will require special attention throughout this article. Let $D_{min}, D_{max}$ and $D_\epsilon$ be the subsets of $D$ defined by
	\begin {itemize}
		\item $D_{min} = \{(p_i, a_i, s_i)_{i\ge 0} \in D;~\forall i\in \mathds{N}, p_i = \epsilon\}$
		\item $D_{max} = \{(p_i, a_i, s_i)_{i\ge 0} \in D;~\forall i\in \mathds{N}, s_i = \epsilon\}$
		\item $D_\epsilon = \{(p_i, a_i, s_i)_{i\ge 0} \in D;~(\exists i_0\in \mathds{N};~\forall i\ge i_0, p_i=\epsilon)$
		or $(\exists i_0\in \mathds{N};~\forall i\ge i_0, s_i=\epsilon)$\}.
	\end{itemize}
	\begin{thm}[\cite{CanSie}]
	Let $\varphi$ be an $A_N$-positive primitive automorphism, $\Sv$ its attracting subshift, $\Svp\subset \Sv$ its
	set of periodic points, and $\rho_{\varphi}:\Sv\to D$ its prefix-suffix development map. The sets $D, D_{min}, D_{max}$ and $D_\epsilon$
	are defined as above.

	The map $\rho_{\varphi}$ is continuous and onto. It is one-to-one except on $\bigcup\limits_{n\in \mathds{Z}} S^n(\Svp)$ and
	it verifies
	\begin{itemize}
		\item $\rho_{\varphi}(\Svp) = D_{min}$ and $\rho_{\varphi}^{-1}(D_{min}) = \Svp$,
		\item $\rho_{\varphi}(S^{-1}(\Svp)) = D_{max}$ and $\rho_{\varphi}^{-1}(D_{max}) = S^{-1}(\Svp)$,
		\item $\rho_{\varphi}(\bigcup\limits_{n\in \mathds{Z}} S^n(\Svp)) = D_\epsilon$ and $\rho_{\varphi}^{-1}(D_\epsilon) = \bigcup\limits_{n\in \mathds{Z}} S^n(\Svp)$.
	\end{itemize}
	\end{thm}

	We keep the notations of the theorem above. As we will see later on, it is important to understand the actions of $\varphi$ and $S$
	on prefix-suffix developments. The action of $\varphi$ is straight forward and can be easily verified. If $W\in \Sv$ and $\rho_{\varphi}(W) = (p_i, a_i, s_i)_{i\ge 0}$,
	then $\rho_{\varphi}(\partial^2\varphi(W)) = (\epsilon, b, r)(p_{i-1}, a_{i-1}, s_{i-1})_{i\ge 1}$ with $\varphi(a_0)=b*r$.

	Studying the action of $S$ on developments involves being careful with empty prefixes or suffixes. In particular, we can not deduce the development
	of $S(W)$ (resp. $S^{-1}(W)$) from $\rho_{\varphi}(W)$ if the latter is in $D_{max}$ (resp. $D_{min}$).

	Let $W\in \Sv$ be such that $\rho_{\varphi}(W) = (p_i, a_i, \epsilon)_{0\le i < i_0}(p_i, a_i, s_i)_{i\ge i_0}$ with $s_{i_0}\ne \epsilon$. We decompose $s_{i_0}$
	into $s_{i_0}=a_{i_0}'*s_{i_0}'$ where $a_{i_0}'$ is a letter of $A_N$. Then we have
	\begin{center}
		$\rho_{\varphi}(S(W)) = (\epsilon, a_i', s_i')_{0\le i < i_0}(p_{i_0}a_{i_0}, a_{i_0}', s_{i_0}') (p_i, a_i, s_i)_{i > i_0}$,
	\end{center}
	where for all $0\le j < i_0$, $\varphi(a_{j+1}') = a_j'*s_j'$.

	If $W$ is such that $\rho_{\varphi}(W) = (\epsilon, a_i, s_i)_{0\le i < i_0}(p_i, a_i, s_i)_{i\ge i_0}$ with $p_{i_0}\ne \epsilon$, we decompose
	$p_{i_0}$ into $p_{i_0}=p_{i_0}'*a_{i_0}'$ where $a_{i_0}'$ is a letter of $A_N$. The development of $S^{-1}(W)$ is then given by
	\begin{center}
		$\rho_{\varphi}(S^{-1}(W)) = (p_i', a_i', \epsilon)_{0\le i < i_0}(p_{i_0}', a_{i_0}', a_{i_0}s_{i_0}) (p_i, a_i, s_i)_{i > i_0}$,
	\end{center}
	where for all $0\le j < i_0$, $\varphi(a_{j+1}') = p_j'*a_j'$.

%	The following proposition is easily obtained by induction, and can be used to recognize
%	points of a common $S$-orbit.
%	\begin{prop}\label{prop:nthshift}
%	Let $W = (U, V)$ be a point of $\Sv$. If $\rho_{\varphi}(W) = (p_i, a_i, s_i)_{i\ge 0}$, then
%	\begin{itemize}
%		\item for all $n\in \mathds{N}^*$ such that $s_n\ne \epsilon$, the word $v = a_0s_0\varphi(s_1)\dots \varphi^{n-1}(s_{n-1})$
%		is a prefix of $V$ and we have
%		$\rho_{\varphi}(S^{|v|}(W)) = (\epsilon, b_i, r_i)_{0\le i < n}(q_n, b_n, r_n)(p_i, a_i, s_i)_{i>n}$ with $q_n=p_na_n$,
%		\item for all $n\in \mathds{N}^*$ such that $p_n\ne \epsilon$, the word $u^{-1} = p_0^{-1}\varphi(p_1^{-1})\dots \varphi^{n-1}(p_{n-1}^{-1})$
%		is a prefix of $U$ and we have
%		$\rho_{\varphi}(S^{-(|u|+1)}(W)) = (q_i, b_i, \epsilon)_{0\le i < n}(q_n, b_n, r_n)(p_i, a_i, s_i)_{i>n}$ with $q_nb_n=p_n$.
%	\end{itemize}
%	\end{prop}

\section{The attracting subshift contains all the information needed to compute the FO-index}\label{sec:attsubcontainsall}

The aim of this section is to explain how the FO-index can be obtained by studying the attracting subshift alone.
We will start by showing that if an automorphism $\psi$ has positive index,
then there is an automorphism $\psi'$ in the isogredience class of a power $\psi^h$ ($h>0$) of $\psi$ such that $\partial^2\psi'$
has at least two attracting points in the attracting subshift. This will lead to the introduction of the important concept of singularity.
A singularity is a set of points of the attracting subshift that are fixed by a common homeomorphism.
The current section ends with two new formulas for the FO-index, both of which only require singularities.

These formulas are the base of an algorithm able to determine the FO-index of an automorphism. In section \ref{sec:identsing},
we explain how the prefix-suffix representation can be used to find the singularities in a finite time.\\

From now on, defining an $A_N$-positive primitive automorphism $\varphi$ implicitly defines its attracting
subshift $\Sv$, the projection $\Sm$ (resp. $\Sp$) of $\Sv$ on its first (resp. second) coordinate, the
half lamination $L(\varphi) = F_N.(\Sm\cup \Sp)$, the
outer class $\Phi^k$ of $\varphi^k$ for all $k\in \mathds{N}^*$ and the set $\Phi_\infty = \bigcup\limits_{k\in \mathds{N}^*} \Phi^k$.

We assume $\varphi$ is an $A_N$-positive primitive automorphism all along section \ref{sec:attsubcontainsall}.

	\subsection{Automorphisms with attracting points outside the half lamination have negative indices}
	For an automorphism $\psi$ of $\Phi_\infty$,
	we show in this section that if $\partial \psi$ has an attracting point $X$ in $\partial F_N\setminus L(\varphi)$ then
	$X$ is its only attracting point, and conclude $\psi$ has negative index.

	We define illegal pairs, which will play an important part in our study of attracting points.
	\begin{defn}
	Let $a$ and $b$ be two elements of $A_N\cup A_N^{-1}$. We say that $ab$ is an \textbf{illegal pair} if
	$|\varphi^k(ab)| < |\varphi^k(a)| + |\varphi^k(b)|$ for some $k\in \mathds{N}$. In other words,
	applying $\varphi$ to $ab$ yields cancellations.
	\end{defn}

	Illegal pairs are a recurring problem in the study of the dynamics of automorphism of the free group. In particular,
	they are ground for the introduction of the train-track representatives, defined in \cite{BH}
	by M. Bestvina and M. Handel. Their role in this article is explained in the following lemma.
	\begin{lem}\label{lem:genwordspart1}
	Let $v = v_0\dots v_h$ be a non empty word of $F_N$ with no illegal pairs and such that $\varphi(v) = p*v*s$ with $p,s\in F_N$ and $s$ non empty.
	The point $Y = \lim\limits_{n\to +\infty} s\varphi(s)\dots \varphi^{n}(s)$ is in $\Sm\cup \Sv$.
	\end{lem}
	\begin{proof}
	In this proof, we will say a letter $v_i$ of $v$ is \textbf{singular} if
	$\varphi(v_0\dots v_{i-1})$ is a prefix of $pv_0\dots v_{i-1}$ and $\varphi(v_{i+1}\dots v_h)$
	is a suffix of $v_{i+1}\dots v_hs$. Observe that if $v_i$ is singular, then it is also a letter of $\varphi(v_i)$.

	First, remark there needs to be at least one singular letter:
	if there were none, the word $\varphi(v_0)$ would be a prefix of $p$ and $\varphi(v_0\dots v_j)$
	would be a prefix of $pv_0\dots v_{j-1}$ for any $1\le j\le p$, which is impossible.
	Note that, as $\varphi$ is primitive, the word $s$ is empty if and only if $v_h$ is the only singular
	letter and verifies $\varphi(v_h) = q*v_h$ for some non empty word $q$. Hence, since we supposed
	$s$ is non empty, if $v_j$ is a singular letter such that $\varphi(v_j) = q'*v_j$, then $j\ne h$ and
	the letter $v_{j+1}$ is also a singular and is such that $\varphi(v_{j+1}) = v_{j+1}*r'$ for some non empty $r'$.\\

	From now on, $v_i$ is a singular letter and we have $\varphi(v_i) = q*v_i*r$ with $r$ non empty.
	Define
	\begin{center}
		$Y = \lim\limits_{n\to +\infty} s\varphi(s)\dots \varphi^{n}(s),$\\
		$Y' = \lim\limits_{n\to +\infty} v_ir\varphi(r)\dots \varphi^{n}(r).$
	\end{center}
	Since $\varphi(v)=p*v*s$, there exists an integer $j$ such that $r\dots \varphi^j(r) = v_{i+1}\dots v_hs*s'$
	for some empty or pure word $s'$. We obtain $r\varphi(r)\dots \varphi^{j+1}(r) = v_{i+1}\dots v_hs\varphi(s)\varphi(s')$,
	and we conclude easily that $Y' = v_i\dots v_{h} Y$.

	The point $Y'$ can easily be seen as a point of $\Sm\cup \Sp$ and we conclude.
	\end{proof}

	We are now able to state the main result of this section. This proposition will allow us to ignore all the points of $\partial F_N\setminus L(\varphi)$
	when searching for the FO-index.
	\begin{prop}\label{prop:outsideL}
		Suppose there is a point $X$ of $\partial F_N\setminus L(\varphi)$ that is attracting for $\partial \psi$ with $\psi = i_w\circ \varphi$.
		Then $w\ne \epsilon$ and we have:
		$$X = \lim\limits_{n\to +\infty} w^{-1}\varphi(w^{-1})\dots \varphi^{n}(w^{-1}).$$
	\end{prop}
	\begin{proof}
		Define $w_{(0)} = \epsilon$ and $w_{(n)} = w^{-1}\varphi(w_{(n-1)})$ for any $n\in \mathds{N}^*$.
		Take a point $V\in \Sm\cup \Sp$ and a prefix $v\in F_N$ of $X$ such that $\lim\limits_{n\to +\infty} \partial \psi^n(vV) = X$.
		For any $n\in \mathds{N}$, define $\psi^n(vV) = x_{(n)} * X_{(n)}$ where $x_{(n)}$ is the largest common prefix
		of $w_{(n)}$ and $\partial \psi^n(vV)$.

		We suppose the sequence $|x_{(n)}|$ does not go to infinity with $n$. Since the sequence
		$(\partial \psi^n(vV))_n$ converges as $n$ goes to infinity, the sequence $x_{(n)}$ then needs to be
		eventually constant. There exists an integer $n_0$ such that for any $n\ge n_0$, we have $x_{(n)} = x$;
		moreover, there exists an integer $h$ such that $x = w_{(h)}r^{-1}$ with $\varphi^{h}(w)=q*r$.
		We deduce that for any $n\ge n_0$, we have $\partial \varphi (X_{(n)}) = \varphi(r)q * X_{(n+1)}$.
		Also, $x$ is such that $X = x*X'$ where $X' = \lim\limits_{n\to \infty} X_{(n)}$. Observe the point $X'$ can
		only contain a finite number of illegal pairs. Indeed, the point $vV$ only contains a finite number of illegal pairs
		and applying $\varphi$ (or $\partial \varphi$) does not increase this number.

		Suppose $X'$ has no illegal pairs. Due to the primitivity of $\varphi$, we can choose a long enough prefix
		$u$ of $X'$ so that $\varphi(u) = \varphi(r)q*u*s$ with $s$ non empty.
		We deduce $X' = \lim\limits_{n\to +\infty} us\varphi(s)\dots \varphi^n(s)$,
		and we conclude with lemma \ref{lem:genwordspart1} that $X\in L(\varphi)$.

		Suppose $X'$ has at least one illegal pair, and suppose $X_{i-1}'X_i'$ is its last illegal pair. Since illegal pairs
		can only appear in the images of illegal pairs, there must be a word $u = X_i'\dots X_j'$ such that
		$\varphi(X_0'\dots X_{i-1}'*u) = \varphi(r)q*X_0'\dots X_{i-1}'*u*s$.
		Again thanks to the primitivity of $\varphi$, we can take $u$ long enough
		so that $s$ is non empty. We deduce $X' = \lim\limits_{n\to +\infty} X_0'\dots X_{i-1}'*us\varphi(s)\dots \varphi^n(s)$,
		and we conclude again with lemma \ref{lem:genwordspart1} that $X\in L(\varphi)$.

		We obtain a contradiction in both cases, an we conclude that the sequence $|x_{(n)}|$ must go to infinity with $n$,
		and that $X$ is as stated.
	\end{proof}

		Recall that $L(\varphi^k) = L(\varphi)$ for any $k\in \mathds{N}^*$. As
		proposition \ref{prop:outsideL} still holds when taking powers of $\varphi$, we obtain the following result.

	\begin{cor}\label{cor:negativeindex}
		If $\psi$ is an automorphism of $\Phi_\infty$ such that $\partial \psi$ has an attracting point outside $L(\varphi)$,
		then $\ind(\psi) = -\frac{1}{2}$.
	\end{cor}
	\begin{proof}
	It is proven in \cite[proposition I.1]{GJLL} that an attracting point can not be a point of $\partial(Fix(\psi))$.
	Hence, the fact that $\partial \psi$ has only one attracting point also implies its fixed subgroup has rank $0$,
	and we conlude.
	\end{proof}

	\subsection{Automorphisms with positive indices}

	For an automorphism $\psi$ with positive index, there are three possible setups we need to study:
	\begin{itemize}
		\item $\psi$ has a trivial fixed subgroup and $\partial \psi$ has at least three attracting points,
		\item $\psi$ has a fixed subgroup of rank (exactly) $1$ and $\partial \psi$ has at least one attracting point,
		\item $\psi$ has a fixed subgroup of rank at least $2$.
	\end{itemize}
	In this section, we study each setup and prove the following theorem.
	\begin{thm}\label{thm:twopoints}
	Let $\varphi$ be an $A_N$-positive primitive automorphism. For any $\psi\in \Phi_\infty$ with a positive index,
	there exists an automorphism $\psi'$ in the isogredience class of $\psi^h$ (for some $h\ge 1$) such that
	$\partial^2 \psi'$ has at least two attracting points in $\Sv$.
	\end{thm}

		\subsubsection{Preliminary results on attracting points of the attracting subshift}

		We still consider $\varphi$ to be an $A_N$-positive primitive automorphism. We start by showing
		that it is equivalent for points of $L(\varphi)$ to be attracting or fixed by a homeomorphism $\partial \psi$
		with $\psi\in \Phi_\infty$. From there on, we will use both terms indifferently. Of course, any attracting
		point of $\partial \psi$ is also fixed, and we only need to prove the converse is true for points of $L(\varphi)$.
		\begin{prop}
		Let $\psi = i_w\circ \varphi^k$ be an automorphism of $\Phi_\infty$. If $X\in L(\varphi)$ is fixed by $\partial \psi$
		(that is $\partial \psi(X)=X$), then $X$ is attracting for $\partial \psi$.
		\end{prop}
		\begin{proof}
		The bounded cancellation theorem, stated in \cite{Coo} and credited there to Grayson and Thurston, tells us
		there exists an integer $C(\varphi)$ such that
		$$|\varphi(u*v)|\ge |\varphi(u)|+|\varphi(v)|-C(\varphi)$$
		for all words of the form $u*v$.
			
		Suppose $X = u*V$ with $V\in \Sm\cup \Sp$. Since $X$ is fixed by $\partial \psi$ and $V$ is pure,
		there exists an integer $i$ such that, for any prefix $v_{(0)}$ of $V$ with $|v_{(0)}|>i$, we have
		$\psi(u*v_{(0)}) = u*v_{(0)}'$ and $u*v_{(0)}'$ is a prefix of $X$.
		Since the automorphism $\varphi$ is $A_N$-positive and primitive, there exists an integer $j>i$
		such that, for any prefix $v_{(1)}$ of $V$ with length $|v_{(1)}|>j$, we have $\psi (u*v_{(1)}) = u*v_{(1)}*v_{(1)}'$ with
		$|v_{(1)}'| > C(\varphi)$, and we conclude.
		\end{proof}

		As suggested by the results of the previous section, any point $X$ that is attracting for
		$\partial \psi$, where $\psi$ is an automorphism of $\Phi_\infty$ with a positive index, needs to be a point of $L(\varphi)$.
		Our aim is to restrain even further the set in which we will look for attracting points to $\Sv$.
		This will be achieved by carefully choosing the representatives of the isogredience classes
		that we want to study.

		\begin{prop}\label{prop:isomove}
		Let $\psi$ is an automorphism of $\Phi_\infty$. If $\partial \psi$ has an attracting point $vV$,
		with $v\in F_N$ and $V\in \Sp\cup \Sm$, then $V$ is an attracting point of $\partial (i_v\circ \psi\circ i_{v^{-1}})$.\\
		Similarly, if the word $vu\in F_N$ is fixed by $\psi$, then the word $uv$ is fixed by $i_v\circ \psi\circ i_{v^{-1}}$.
		\end{prop}

		In addition, we prove the following proposition in order to make the transition to points of $\partial^2 F_N$.
		\begin{prop}\label{prop:uni-bi}
		Let $\psi = i_w\circ \varphi^k$ be an automorphism of $\Phi_\infty$, let $V\in \Sp$ (resp. $U\in \Sm$)
		be an attracting point of $\partial \psi$ and let $U$ (resp. $V$) be such that $(U, V)$ is a point of $\Sv$.
		The point $(U, V)$ is an attracting point of $\partial^2 \psi^h$ for some $h\ge 1$.
		\end{prop}
		\begin{proof}
		The main argument is there exist $j\in \mathds{N}^*$ and $p\in \mathds{Z}$ such that $\partial^2 \varphi^{kj}(U, V) = S^p(U, V)$.
		This is a consequence of a theorem of \cite{Que} stating that the canonical projection of $\Sv$ to $\Sm$ (resp. $\Sp$)
		is finite-to-one. If $\psi=\varphi^k$, the integer $p$ needs to be $0$ and we are done.
		If $\psi\ne \varphi^k$, note that $p\ne 0$ and define $w'^{-1}$ to be the prefix of length $|p|$ of $U$ (resp. $V$) if $p>0$ (resp. $p<0$).
		The point $(U, V)$ is then attracting for $\partial^2 (i_{w'}\circ \varphi^{kj})$ and the automorphism $i_{w'}\circ \varphi^{kj}$
		is necessarily a power of $\psi$.
		\end{proof}

		Note that, as we are ultimately interested in the FO-index, studying powers of automorphisms has no influence on the outcome.

		We now state a few results on automorphisms $\psi = i_w\circ\varphi^k\in \Phi_\infty$ such that $\partial^2 \psi$ has an attracting point in $\Sv$.
		As we will see in section \ref{sec:identsing}, the word $w$ plays a major part in the search
		for automorphisms with positive indices. We start here with some preliminary results.
		\begin{prop}\label{prop:purity}
		Let $\psi = i_w\circ \varphi^k$, with $w\ne \epsilon$, be an automorphism of $\Phi_\infty$.
		If $\partial^2 \psi$ fixes a point $W\in \Sv$, then $w$ is pure.
		\end{prop}
		\begin{proof}
		If $\partial^2 \psi$ fixes $W = (U, V)$,
		then $V = w^{-1}\partial \varphi^{k}(V)$ and $wV = \partial \varphi^{k}(V)$.
		Since $V$ is in $\Sp$, then so are $\varphi^{k}(V)$ and $wV$.
		We are left with three options:
		\begin{itemize}
			\item $w$ is pure positive,
			\item $w$ is pure negative (in which case $w^{-1}$ is a prefix of $V$),
			\item $w = w_{(p)}w_{(n)}$ with $w_{(p)}$ pure positive and $w_{(n)}$ pure negative.
		\end{itemize}
		Similarly, $U = w^{-1}\partial \varphi^{k}(U)$, $wU\in \Sm$ and $w$ verifies one of the following conditions:
		\begin{itemize}
			\item $w$ is pure positive (in which case $w^{-1}$ is a prefix of $U$),
			\item $w$ is pure negative,
			\item $w = w_{(m)}w_{(q)}$ with $w_{(m)}$ pure positive and $w_{(q)}$ pure negative.
		\end{itemize}
		Hence, if $w$ is not pure, then $w=w_{(p)}w_{(n)}$ with $w_{(p)}$ pure positive and $w_{(n)}$ pure negative
		and $w=w_{(m)}w_{(q)}$ with $w_{(m)}$ pure negative and $w_{(q)}$ pure positive, which is impossible. 
		\end{proof}

		We insist on the fact that $w$ can not be any word of $F_N$. In addition to its purity,
		the above proposition tells us $w^{-1}$ is a prefix of $U$ (resp. $V$) if $w$ is pure positive (resp. negative),
		so that $wU$ (resp. $wV$) is always a point of $\Sm$ (resp. $\Sp$).
		We deduce the following proposition.

		\begin{prop}\label{prop:freetocomb}
		Let $\psi = i_w\circ \varphi^k$, with $w\ne \epsilon$, be an automorphism of $\Phi_\infty$ such that
		$W\in \Sv$ is a fixed point of $\partial^2 \psi$.
		\begin{itemize}
			\item If $w$ is pure positive, then $S^{|w|}\circ\partial^2 \varphi^{k}(W) = W$.
			\item If $w$ is pure negative, then $S^{-|w|}\circ\partial^2 \varphi^{k}(W) = W$.
		\end{itemize}
		\end{prop}

		Finally, we give a little insight on automorphisms $\psi$ such that $\partial^2 \psi$ fixes more than one point of $\Sv$.
		The proposition below is trivially verified and is one of the fundamental ideas of the algorithm to come.
		\begin{prop}\label{prop:commoncoord}
		Let $\psi = i_w\circ \varphi^k$, with $w\ne \epsilon$, be an automorphism of $\Phi_\infty$.
		Let $W = (U, V)$ and $W' = (U', V')$ be two points of $\Sv$ fixed by $\partial^2 \psi$.
		If $w$ is pure positive, then we have
		\begin{center}
			$U = U' = \lim\limits_{n\to +\infty} w^{-1}\varphi^{k}(w^{-1})\varphi^{2k}(w^{-1})\dots \varphi^{(n-1)k}(w^{-1})\varphi^{nk}(w^{-1})$.
		\end{center}
		If $w$ is pure negative, then
		\begin{center}
			$V = V' = \lim\limits_{n\to +\infty} w^{-1}\varphi^{k}(w^{-1})\varphi^{2k}(w^{-1})\dots \varphi^{(n-1)k}(w^{-1})\varphi^{nk}(w^{-1})$.
		\end{center}
		\end{prop}
		Note that this is not always true if $w=\epsilon$.\\

		Recall we want to prove theorem \ref{thm:twopoints}; namely, we want to be able to identify automorphisms
		with positive indices with automorphisms fixing several points of the attracting subshift.
		This means we will need to find new points of the attracting subshift that are attracting for a given automorphism.
		This is the purpose of the following proposition.
		\begin{prop}\label{prop:genwordspart2}
		Let $v=v_0v_1\dots v_h\in F_N$ be a non empty pure word such that $\varphi(v) = p*v*s$ with $p,s\in F_N$ possibly empty.
		There exists $k\ge 1$ such that the homeomorphism $\partial^2 (i_{p}\circ \varphi)^k$ has an attracting point $(U, V)$ in
		$\Sv$ verifying $V_0=v_0$ (resp. $U_0=v_0$) if $v$ is pure positive (resp. negative).
		\end{prop}
		\begin{proof}
		Since $v$ is pure and $\varphi$ is $A_N$-positive, the word $v$ does not contain any illegal pair.

		First suppose $p$ is empty. This means $v_0$ is the first letter of $\varphi(v_0)$.
		The point $X = \lim\limits_{n\to +\infty} \varphi^n(v_0)$ is then a point of $\Sm\cup \Sp$
		that is fixed by $\partial \varphi$, and we conclude with proposition \ref{prop:uni-bi}.

		We now assume $p$ is not empty. We actually work with $\varphi(v^{-1}) = s^{-1}*v^{-1}*p^{-1}$.
		Recall that a letter $v_i^{-1}$ of $v^{-1}$ is singular if $\varphi(v_{i-1}^{-1}\dots v_0^{-1})$
		is a suffix of $v_{i-1}^{-1}\dots v_0^{-1}p^{-1}$ and $\varphi(v_h^{-1}\dots v_{i+1}^{-1})$
		is a prefix of $s^{-1}v_h^{-1}\dots v_{i+1}^{-1}$.

		Define $X' = \lim\limits_{n\to +\infty} p^{-1}\varphi(p^{-1})\dots \varphi^n(p^{-1})$.
		Following the proof of lemma \ref{lem:genwordspart1} provides a singular letter $v_i^{-1}$
		of $v^{-1}$ such that $\varphi(v_i^{-1}) = r^{-1}*v_i^{-1}*q^{-1}$ with $q^{-1}$ non empty
		and such that the point $X = v_i^{-1}\dots v_0^{-1}X'$ is in $\Sm\cup \Sp$.
		If $v$ is pure positive (resp. negative) then choose a point $Y$ of $\Sp$ (resp. $\Sm$)
		such that $(X, Y)$ (resp. $(Y, X)$) is in $\Sv$, and define $(U, V) = S^{-(i+1)}(X, Y)$
		(resp. $(U, V) = S^{i+1}(Y, X)$). Observe that $U = X'$ (resp. $V = X'$) if $p$ is pure positive (resp. negative)
		and conclude with proposition \ref{prop:uni-bi} that the point $(U, V)\in \Sv$ is fixed by
		$\partial^2 (i_{p}\circ \varphi)^k$ for some integer $k\ge 1$ and verifies $V_0=v_0$ (resp. $U_0=v_0$)
		if $v$ is pure positive (resp. negative).
		\end{proof}

		We now move on to proving theorem \ref{thm:twopoints} for each of the three possible cases
		of automorphisms with positive indices.

		\subsubsection{Automorphisms with trivial fixed subgroups}

		In this section, we assume $\psi = i_w\circ \varphi^k$
		is an automorphism with a trivial fixed subgroup and a positive index.
		Using propositions \ref{prop:isomove} and \ref{prop:uni-bi}, we can simply assume $\partial^2 \psi$ has an attracting
		point $(U, V)$ in $\Sv$ and $\partial \psi$ has an other attracting point $X$ ($X\ne U$ and $X\ne V$).
		As $\ind(\psi) > 0$, we must have $X=v*V'$ for some $v\in F_N$ and $V'\in \Sm\cup \Sp$.

		Obviously, if $v$ is empty, then proposition \ref{prop:uni-bi} tells us some power of $\partial^2 \psi$ fixes
		at least two points of $\Sv$. We suppose $v\ne \epsilon$.
		We also assume $v_0\ne U_0$ and $v_0\ne V_0$ (where $v_0,~U_0,~V_0$ are the first letters of $v,~U$ and $V$ respectively);
		if it were not the case, we could simply find an automorphism in the isogredience class of $\psi$ verifying this property.
		The main point is the following.

		\begin{prop}\label{prop:case1}
		Let $v_0$ be the first letter of $v$. There is a point
		$(U'', V'')\in \Sv$ that is attracting for $\partial^2 \psi^h$ for some integer $h\ge 1$ and such that $v_0=U_0''$
		(resp. $v_0=V_0''$) if $v_0\in A_N^{-1}$ (resp. $v_0\in A_N$).
		\end{prop}
		\begin{proof}
		Define $v = v_{(0)}*v_{(1)}*\dots *v_{(p)}$, where for all $0\le i\le p$, $v_{(i)}$ is non empty and pure
		and for all $0\le i\le p-1$, $v_{(i+1)}$ is pure negative (resp. positive) if $v_{(i)}$ is pure positive (resp. negative).
		Since $\varphi$ is $A_N$-positive, the number of orientation changes (pairs $v_jv_{j+1}$ where $v_j\in A_N$ (resp. $v_j\in A_N^{-1}$)
		if $v_{j+1}\in A_N^{-1}$ (resp. $v_{j+1}\in A_N$)) can not increase when applying $\varphi$.

		We assume $v_{(0)}$ is pure positive; a similar approach can be used to prove the result when $v_{(0)}$ is pure negative.
		Obviously, the point $v*V'$ has a finite number of orientation changes and applying $\partial \psi$ to $v*V'$ does not
		change this number. Hence, considering $w$ is either empty or pure by proposition \ref{prop:purity}, we must have
		$\partial \varphi^k (v*V') = v_{(0)}'*v_{(1)}*\dots *v_{(p)}*V'$, where $v_{(0)}'$ is one of the following:
		\begin{itemize}
			\item $v_{(0)}' = \epsilon$,
			\item $v_{(0)}'$ is pure negative,
			\item $v_{(0)}'$ is pure positive and is a strict suffix of $v_{(0)}$,
			\item $v_{(0)}'$ is pure positive and $v_{(0)}$ is a suffix of $v_{(0)}'$.
		\end{itemize}
		We need to have $w^{-1}v_{(0)}' = v_{(0)}$.
		In the first three cases, the word $w^{-1}$ has to be non empty and pure positive.
		This means the first letter of $v_{(0)}$ must be equal to the first
		letter of $w^{-1}$. Since $\partial \psi (V) = V = w^{-1}*\partial \varphi^k(V)$, we obtain $v_0=V_0$, which contradicts
		our hypothesis.

		For the fourth case to be possible, the word $w^{-1}$ needs to be either empty or pure negative and
		we must have $\varphi^k(v_{(0)}) = w * v_{(0)} * s$, where $s$ is either empty or pure positive.
		We conclude with proposition \ref{prop:genwordspart2}.
		\end{proof}

		\subsubsection{Automorphisms with fixed subgroups of rank $1$}

		We start with a preliminary result regarding elements of fixed subgroups.
		\begin{prop}\label{prop:notpure}
		Let $\psi=i_w\circ \varphi^k$ be an automorphism of $\Phi_\infty$ and let $u$ be a non empty word of $F_N$
		such that $\psi(u) = u$. Then the word $u$ is not pure.
		\end{prop}
		\begin{proof}
		Suppose $u$ is pure and choose an integer $n$ such that $|u^n| > |w|$. The word $u$ is fixed by $\psi$,
		and the automorphism $\varphi$ is $A_N$-positive; this implies the word $\varphi^k(u^nu^n) = wu^nu^nw^{-1}$
		is also pure. This can only happen if $w$ is either empty or pure and if we have either $u^n = w^{-1}*u'$ for some non empty $u'\in F_N$
		or $u^n = u''*w$ for some non empty $u''\in F_N$. In both cases, we have $|wu^nu^nw^{-1}| = |u^nu^n|$.

		Recall that $\varphi$ is primitive. We can assume $k$ is large enough so that $|\varphi^k(u)| > |u|$ (otherwise
		we work with powers of $\psi$). We obtain $|\varphi^k(u^nu^n)| > |u^nu^n|$, which is impossible.
		\end{proof}

		For the end of this section, $\psi = i_w\circ \varphi^k$ is an automorphism of $\Phi_\infty$ with
		a fixed subgroup of rank (exactly) $1$ and (at least) one attracting point
		$X = v*V'$, with $v\in F_N$ and $V'\in \Sm\cup \Sp$. Let $u$ be a base of its fixed subgroup. As any automorphism
		in the isogredience class of $\psi$ also has a fixed group of rank (exactly) $1$ and has as many attracting points as $\psi$,
		we can assume (proposition \ref{prop:isomove}) $\psi$ was chosen so that $u$ is cyclically reduced (if $u_0$ and $u_p$ are its
		first and last letter respectively, we have $u_0\ne u_p^{-1}$), and $u_p\in A_N$ (resp. $u_p\in A_N^{-1}$) if $u_0\in A_N^{-1}$ (resp. $u_0\in A_N$)
		(this last condition is possible because $u$ is not pure).

		\begin{lem}\label{lem:case2purity}
		The word $w$ is empty or pure positive (resp. empty or pure negative) if $u_0\in A_N$ (resp. $u_0\in A_N^{-1}$).
		\end{lem}
		\begin{proof}
		Choose an integer $n$ large enough so that we have both $|u^n| > |w|$ and $u^nu^{n}vV' = u^n*u^nvV'$.
		Since $u$ is fixed by $\psi$, we have $\varphi^k(u^nu^n) = w~u^nu^n~w^{-1}$. Moreover, since
		$u$ is cyclically reduced, we actually have either $\varphi^k(u^nu^n) = w*u^n*u^n~w^{-1}$
		or $\varphi^k(u^nu^n) = w~u^n*u^n*w^{-1}$ (or possibly both). We assume $\varphi^k(u^nu^n) = w*u^n*u^n~w^{-1}$
		(otherwise we work with $u^{-n}$).

		Both $vV'$ and $u^n*u^nvV'$ are attracting for $\partial \psi$, and we have
		$\partial \varphi^k(u^n*u^nvV') = w*u^n*u^nvV'$. Since $V'\in \Sm\cup \Sp$, there is a finite number of orientation changes
		in (the reduced form of) $u^n*u^nvV'$. On the one hand, applying $\partial \varphi$ can not increase the number of orientation changes
		since $\varphi$ is $A_N$-positive; on the other hand, $w*u^n*u^nvV'$ can not have less orientation changes than $u^n*u^nvV'$.
		We deduce $u^n*u^nvV'$ and $w*u^n*u^nvV'$ have the same number of orientation changes and conclude.
		\end{proof}

		Recall that $u_p$ is the last letter of $u$ and that $u_p\in A_N$ (resp. $u_p\in A_N^{-1}$) if $u_0\in A_N^{-1}$ (resp. $u_0\in A_N$).
		Define $u_{(0)}$ (resp. $u_{(q)}$) as the longest pure prefix (resp. suffix) of $u$. We get from lemma \ref{lem:case2purity}
		that $\varphi^k(u) = w*u*w^{-1}$ contains as many orientation changes as $u$
		and deduce $\varphi^k(u_{(0)}) = w*u_{(0)}*s$ for some empty or pure word $s$ and $\varphi^k(u_{(q)}^{-1}) = w*u_{(q)}^{-1}*s'$
		for some empty or pure word $s'$. Applying proposition \ref{prop:genwordspart2} twice then yields the following proposition.
		\begin{prop}\label{prop:case2}
		There exists an integer $h\ge 1$ such that $\partial^2 \psi^h$ has two attracting points $(U, V)$ and $(U'', V'')$ in $\Sv$
		such that $V_0 = u_0$ and $V_0'' = u_p^{-1}$ (resp. $U_0 = u_0$ and $U_0'' = u_p^{-1}$) if $u_0\in A_N$ (resp. $u_0\in A_N^{-1}$).
		\end{prop}

		Observe that the word $u$ is also fixed by the automorphism $i_{wu^h}\circ \varphi^k$ for any integer $h$, although
		there are only a finite number of value for $h$ that yield an automorphism with a positive index ($h=0$ possibly being the only one).
		Interestingly enough, with lemma \ref{lem:case2purity} giving $wu^h = w*u^h$, one can easily prove, for any $h\ne 0$,
		$$\lim\limits_{n\to +\infty}u^{-h}w^{-1}\varphi^k(u^{-h}w^{-1})\dots \varphi^{kn}(u^{-h}w^{-1}) = \lim\limits_{n\to +\infty} u^{-n}.$$
		Note however, that such a point is not attracting (\cite[proposition I.1]{GJLL}).

		\subsubsection{Automorphisms with fixed subgroups of rank at least $2$}

		As in the previous section, the idea is to choose a correct automorphism to work with. However,
		with a fixed subgroup of rank more than $1$, the existence of such an automorphism is not as obvious,
		and requires a preliminary work.
		\begin{lem}
		Let $\psi_0$ be an automorphism of $\Phi_\infty$ with a fixed subgroup of rank at least $2$.
		There exists an automorphism $\psi$ in the isogredience class of $\psi_0$ and a word $u=u_0\dots u_p\in F_N$ such that:
		\begin{itemize}
			\item $(1)$ $u$ is fixed by $\psi$,
			\item $(2)$ $u$ is cyclically reduced,
			\item $(3)$ $u_p\in A_N$ (resp. $u_p\in A_N^{-1}$) if $u_0\in A_N^{-1}$ (resp. $u_0\in A_N$),
			\item $(4)$ no strict prefix or suffix of $u$ is fixed by $\psi$.
		\end{itemize}
		We will say $u$ is $\boldsymbol{\psi}$\textbf{-reduced}.
		\end{lem}
		\begin{proof}
		Choose a word $u_{(0)}\in F_N$ that is fixed by $\psi_0$. Suppose $u_{(0)}$ is not $\psi_0$-reduced.
		It is easy to find a pair $(\psi_1, u_{(1)})$ that will verify the first three conditions (proposition \ref{prop:isomove}).
		If the fourth is also verified, we are done. If not, let $u_{(2)}$ be a strict prefix or suffix of $u_{(1)}$ that
		is fixed by $\psi_1$. The point here is that $|u_{(2)}| < |u_{(0)}|$. Hence, iterating the process
		will effectively yield a pair $(\psi, u)$ satisfying all four conditions.
		\end{proof}

		For the end of this section, $\psi = i_w\circ \varphi^k$ is an automorphism of $\Phi_\infty$ with a fixed subgroup of rank at least $2$.
		We suppose $\psi$ was chosen so that there exists a $\psi$-reduced word $u=u_0\dots u_p$. We choose a non empty word $v$ that is fixed by $\psi$
		and such that neither $u$ nor $u^{-1}$ is a prefix or a suffix of $v$.
		\begin{lem}\label{lem:case3safety}
		There exists a non empty word $v'''\in F_N$ such that $uvu^{-1} = u_0*v'''*u_0^{-1}$.
		\end{lem}
		\begin{proof}
		Since $u$ is $\psi$-reduced and since $u^{-1}$ is not a prefix of $v$, there exists a non empty suffix $v'$ of $v$ and
		a positive integer $i$ such that $uv = u_0\dots u_{i-1}*v'$. Since $u$ is not a suffix of $v$, there is a
		possibly empty prefix $v''$ of $v'$ and a positive integer $j$ such that $v'u^{-1} = v''*u_{j-1}^{-1}\dots u_0^{-1}$.
		We conclude immediately if $v''$ is non empty. If $v''$ is empty, we first note that $uvu^{-1}$ can not be empty as
		$v\ne \epsilon$. We also recall that since $uvu^{-1}$ is fixed by $\psi$, it can not be a strict prefix or suffix of $u$
		(because $u$ is $\psi$-reduced) and we conclude.
		\end{proof}

		We are now able to proceed as in the previous section.
		\begin{lem}\label{lem:case3purity}
		The word $w$ is empty or pure positive (resp. empty or pure negative) if $u_0\in A_N$ (resp. $u_0\in A_N^{-1}$).
		\end{lem}
		\begin{proof}
		Choose an integer $n$ large enough so that we have $|u^n| > 2|w|$.
		Since $u$ is fixed by $\psi$, we have $\varphi^k(u^n) = w~u^n~w^{-1}$. Moreover, since
		$u$ is cyclically reduced, we actually have either $\varphi^k(u^n) = w*u^n~w^{-1}$
		or $\varphi^k(u^n) = w~u^n*w^{-1}$ (or possibly both). We assume $\varphi^k(u^n) = w*u^n~w^{-1}$
		(otherwise we work with $u^{-n}$). Using lemma \ref{lem:case3safety} then yields
		\begin{center}
			$\varphi^k(u^n*u_0*v'''*u_0^{-1}*u^{-n}) = w*u^n*u_0*v'''*u_0^{-1}*u^{-n}*w^{-1}$.
		\end{center}

		Again, applying $\varphi$
		can not increase the number of orientation changes, and the word $w*u^n*u_0*v'''*u_0^{-1}*u^{-n}*w^{-1}$ can not
		have less orientation changes than $u^n*u_0*v'''*u_0^{-1}*u^{-n}$. Hence, both words
		have the same number of orientation changes, and we conclude.
		\end{proof}

		Everything else is copied from the previous section, and we obtain the same proposition for this last type
		of automorphism.
		\begin{prop}\label{prop:case3}
		There exists an integer $h\ge 1$ such that $\partial^2 \psi^h$ has two attracting points $(U, V)$ and $(U'', V'')$ in $\Sv$
		such that $V_0 = u_0$ and $V_0'' = u_p^{-1}$ (resp. $U_0 = u_0$ and $U_0'' = u_p^{-1}$) if $u_0\in A_N$ (resp. $u_0\in A_N^{-1}$).
		\end{prop}

	\subsection{The singularity formula}

	We still assume $\varphi$ is an $A_N$-positive primitive automorphism, $\Sv$ is its attracting subshift,
	$\Phi^k$ is the outer class of $\varphi^k$ and $\Phi_\infty = \{i_w\circ \varphi^k;~w\in F_N,~k\in \mathds{N}^*\}$.

	Recall that $\ind(\Phi^k) = \sum\limits_{[\psi]} \max (0, \ind(\psi))$, where the sum is taken over all isogredience
	classes of $\Phi^k$. Choose an integer $k$ such that $\ind(\Phi^k) = \ind(\Phi_\infty)$. For each isogredience
	class of automorphisms with positive indices, propositions \ref{prop:case1}, \ref{prop:case2} and \ref{prop:case3} give us
	a representative $\psi$ such that $\partial^2 \psi$ fixes several points $\Sv$. The set of points of $\Sv$
	fixed by $\partial^2 \psi$ is defined as a singularity.

	\begin{defn}
	A \textbf{singularity} is a set $\Omega$ of (pairwise distinct) points of $\Sv$ satisfying the following conditions:
	\begin{itemize}
		\item $\Omega$ contains at least two elements,
		\item there exists an automorphism $\psi\in \Phi_\infty$ such that all the points of $\Omega$
		are fixed points of $\partial^2 \psi$,
		\item for any $h\in \mathds{N}^*$, if $(U, V)\in \Sv$ is a fixed point of $\partial^2 \psi^h$,
		then $(U, V)\in \Omega$,
		\item there exist two points $(U, V)$ and $(U', V')$ of $\Omega$ such that we have either
		$V_0\ne V_0'$ or $U_0\ne U_0'$ or both.
	\end{itemize}
	We say that the automorphism $\psi$ fixes the singularity $\Omega$.

	For any singularity $\Omega$, we define $H_\Omega$ (for Halves of points of $\Omega$) as the set
	of coordinates of points of $\Omega$:
	$$
		H_\Omega = \bigcup\limits_{(U, V)\in \Omega} \{U, V\}.
	$$
	\end{defn}

	Obviously, knowing that an automorphism $\psi\in \Phi_\infty$ fixes a singularity is (most often)
	not enough to deduce its index, as the singularity may not contain all the attracting points of $\partial \psi$
	and (apparently) does not tell us anything regarding the fixed subgroup of $\psi$.
	We are going to show that all the missing informations are actually contained in other singularities,
	which will be fixed by automorphisms in the isogredience class of $\psi$.

	Also, the reader should keep in mind that defining an automorphism which fixes a singularity
	implicitly declares it has a positive index.

		\subsubsection{The singularity graph}

		In \cite{CL}, M. Cohen and M. Lustig give an algorithm able to determine the
		fixed subgroup of an automorphism. The algorithm involves constructing a graph
		whose connected components represent the isogredience classes, and whose
		fundamental group represents the fixed subgroups. In this section, we obtain
		similar results. We adapt the definitions to highlight the importance of
		the singularities and to also include the attracting points. This will
		result in a way to compute the FO-index by only using singularities.

		\begin{defn}
		The \textbf{singularity graph} $\mathcal{G}$ associated to the automorphism $\varphi$ is the graph
		defined by the triplet $(\mathcal{N}, \mathcal{E}_f, \mathcal{E}_i)$ where:
		\begin{itemize}
			\item $\mathcal{N}$ is the set of nodes. It is the set of singularities.
			\item $\mathcal{E}_f$ is the set of finite edges. Finite edges are oriented and each one is
			labeled with a non empty pure positive word of $F_N$. There is a finite edge labeled $u$ from
			$\Omega$ to $\Omega'$, and we write it $(\Omega, \Omega', u)$, if
			\begin{itemize}
				\item there exist $(U, V)\in \Omega$ and $(U', V')\in \Omega'$ such that $(U, V) = u(U', V')$,
				\item for any strict prefix $v$ (resp. suffix $v'$) of $u$, no singularity contains $v^{-1}(U, V)$
				(resp. $v'(U', V')$).
			\end{itemize}
			\item $\mathcal{E}_i$ is the set of infinite edges. We don't require them to be oriented, but
			they are labeled with points of $\Sm\cup \Sp$. An infinite edge is anchored at a unique singularity.
			There is an infinite edge labeled $V\in H_\Omega$ anchored at $\Omega$, and we write it $(\Omega, V)$, if
			\begin{itemize}
				\item for any finite edge $(\Omega, \Omega', u)$, we have $V_0\ne u_0$ (where $V_0$ and
				$u_0$ are the first letters of $V$ and $u$ respectively),
				\item for any finite edge $(\Omega', \Omega, u)$, we have $U_0\ne u_p^{-1}$ (where $u_p$
				is the last letter of $u$).
			\end{itemize}
			It should be noted that if $(\Omega, V)$ and $(\Omega, V')$ are two infinite edges, then $V_0\ne V_0'$.
			Indeed, if $V$ and $V'$ are two points of $H_\Omega$ with a common prefix $u\ne \epsilon$, then
			there exists a finite edge $(\Omega, \Omega', u)$ (if $u$ is pure positive) or $(\Omega', \Omega, u^{-1})$
			(if $u$ is pure negative) and neither $V$ nor $V'$ can label an infinite edge.
		\end{itemize}
		The connected component containing a singularity $\Omega$ will be denoted $\mathcal{G}_\Omega$.
		\end{defn}

		A \textbf{path} in a singularity graph is a tuple $(\Omega_0, \dots, \Omega_h)$ of nodes such that, for any $0\le i < h$,
		there is a finite edge from $\Omega_i$ to $\Omega_{i+1}$ or from $\Omega_{i+1}$ to $\Omega_i$. We will say
		that this path starts at $\Omega_0$ and ends at $\Omega_h$. Note that there can be no
		edge from one singularity to itself. The path is called a \textbf{cycle} if $\Omega_h=\Omega_0$.
		It is a \textbf{trivial} cycle if for any $0\le i\le h$, we have $\Omega_i = \Omega_{h-i}$.

		We also define the \textbf{path map} $\gamma$ of the singularity graph $\mathcal{G}$:
		\begin{itemize}
			\item if $(\Omega_0, \Omega_1, u)$ is a finite edge of $\mathcal{G}$, then define $\gamma(\Omega_0, \Omega_1)=u$
			and $\gamma(\Omega_1, \Omega_0)=u^{-1}$,
			\item if $(\Omega_0, \dots, \Omega_h)$ is a path in $\mathcal{G}$, then define
			$\gamma(\Omega_0, \dots, \Omega_h)=\gamma(\Omega_0, \Omega_1)\dots \gamma(\Omega_{h-1}, \Omega_h)$.
		\end{itemize}
		It is important to note that, from the definition of finite edges, if $(\Omega_0, \Omega_1, \Omega_2)$
		is a path with $\Omega_0\ne \Omega_2$, then we actually have
		$\gamma(\Omega_0, \Omega_1, \Omega_2)=\gamma(\Omega_0, \Omega_1)*\gamma(\Omega_1, \Omega_2)$.
		In particular, if $\gamma(\Omega_0, \dots, \Omega_h)=\epsilon$, then the path $(\Omega_0, \dots, \Omega_h)$ is a trivial cycle.

		Lastly, we loosely use the notation $\pi_1(\mathcal{G})$ to refer to the fundamental group of the topological graph
		associated to $\mathcal{G}$.

		\begin{thm}\label{thm:graphindex}
		Let $\varphi$ be an $A_N$-positive primitive automorphism.
		Let $k$ be an integer such that $\ind(\Phi_\infty) = \ind(\Phi^k)$ (where $\ind(\Phi_\infty)$ is the FO-index of $\varphi$
		and $\Phi^k$ is the outer class of $\varphi^k$) and let $\psi$ be an automorphism
		of $\Phi^k$ that fixes a singularity $\Omega$. Let $\mathcal{G}$ the singularity graph of $\varphi$
		and $\mathcal{G}_\Omega$ the connected component of $\mathcal{G}$ containing $\Omega$.
		\begin{itemize}
			\item $(1)$ We have $rk(\pi_1(\mathcal{G}_\Omega)) = rk(Fix(\psi))$.
			\item $(2)$ The number of infinite edges of $\mathcal{G}_\Omega$ is equal to the number of equivalence
			classes of attracting points of $\partial \psi$.
		\end{itemize}
		\end{thm}
		\begin{proof}
		We start by proving the following lemma.
			\begin{lem}\label{lem:insidethm}
			If $(\Omega, \Omega_1, \dots, \Omega_h)$ is a path in $\mathcal{G}_\Omega$, then for any point $V\in H_{\Omega_h}$,
			the point $\gamma(\Omega, \Omega_1, \dots, \Omega_h)V$ is fixed by $\partial \psi$.
			\end{lem}
			\begin{proof}
			Let $(\Omega, \Omega_1, u)$ (resp. $(\Omega_1, \Omega, v)$) be a finite edge in $\mathcal{G}_\Omega$. We get from the definition
			of finite edges, and from proposition \ref{prop:isomove}, that any point $V$ of $H_{\Omega_1}$ is fixed by $\partial (i_u\circ \psi\circ i_{u^{-1}})$
			(resp. $\partial (i_{v^{-1}}\circ \psi\circ i_v)$). In both cases, any point of $H_{\Omega_1}$ is fixed by
			$\partial (i_{\gamma(\Omega, \Omega_1)}\circ \psi\circ i_{(\gamma(\Omega, \Omega_1))^{-1}})$. Recall for any two words $u_{(0)}$ and $u_{(1)}$ of $F_N$, we have
			$i_{u_{(1)}}\circ i_{u_{(0)}} = i_{u_{(0)}u_{(1)}}$. Hence, by induction, if $(\Omega, \Omega_1, \dots, \Omega_h)$ is
			a path, then any point of $H_{\Omega_h}$ is fixed by
			$\partial (i_{\gamma(\Omega, \Omega_1, \dots, \Omega_h)}\circ \psi \circ i_{(\gamma(\Omega, \Omega_1, \dots, \Omega_h))^{-1}})$,
			and we conclude.
			\end{proof}
			\begin{cor}\label{cor:concom}
			Singularities of $\mathcal{G}_\Omega$ are fixed by automorphisms in the isogredience class of $\psi$. Conversely,
			if $\psi'$ is in the isogredience class of $\psi$ and fixes a singularity $\Omega'$, then $\Omega'$ is a node of $\mathcal{G}_\Omega$.
			\end{cor}
		We now prove proposition $(1)$. It is clear from the lemma above that for any cycle $c$ starting and ending at $\Omega$, the word
		$\gamma(c)$ is fixed by $\psi$.
		
		Conversely, suppose $u$ is a non empty word of $F_N$
		that is fixed by $\psi$. Recall that $\psi$ fixes the singularity $\Omega$; for any point $X$ of $H_\Omega$, the point
		$uX$ is fixed by $\partial \psi$. According to proposition \ref{prop:case1}, there is a point $(U, V)\in \Omega$ and a non empty
		prefix $u_{(0)}$ of $u$ such that $u = u_{(0)}*u'$ and $u_{(0)}$ is a prefix of either $U$ or $V$. Suppose $u_{(0)}$ is the largest prefix of 
		$u$ that is also a prefix of $U$ or $V$. Define $\psi' = i_{u_{(0)}}\circ \psi\circ i_{u_{(0)}^{-1}}$ and $(U', V')=u_{(0)}^{-1}(U, V)$. The point $(U', V')$ is fixed by
		$\partial^2 \psi'$, the point $u'X$ is fixed by $\partial \psi'$, and we have $u'_0\ne V'_0$ and $u'_0\ne U'_0$. We deduce from proposition \ref{prop:case1}
		(and from the fact that $\psi'\in \Phi^k$ with $\ind(\Phi_\infty) = \ind(\Phi^k)$) that $\psi'$ fixes a singularity $\Omega'$.
		Note that due to the definition of finite edges, there might not be a finite edge between $\Omega$ and $\Omega'$;
		there is, however, a path joining them. We can now iterate the process, and each new iteration will consume the word $u$ until we eventually fall back to $\Omega$.
		We obtain a cycle $c$ starting and ending at $\Omega$ such that $\gamma(c) = u\ne \epsilon$, ensuring
		the cycle is not trivial.

		Finally, recall that $\gamma (c) = \epsilon$ (where $c$ is a cycle) if and only if $c$ is a trivial cycle. Hence, if $\{c_0, \dots, c_h\}$
		is a set of independant cycles all starting and ending at $\Omega$, then $\{\gamma(c_0), \dots, \gamma(c_h)\}$ is a set of independant
		generators of $Fix(\psi)$, and conversely. We obtain the following important theorem, which also concludes the proof of proposition $(1)$.
			\begin{thm}\label{thm:insidethm}
			Suppose $rk(\pi_1(\mathcal{G}_\Omega)) = h+1$ and let $\{c_0, \dots, c_h\}$ be a set of independant
			cycles of $\mathcal{G}_\Omega$ all starting and ending at $\Omega$. Then $\{\gamma(c_0), \dots, \gamma(c_h)\}$ is a base of $Fix(\psi)$.
			\end{thm}
		We move on to proposition $(2)$. We start by showing that every equivalence class of attracting points is represented
		by an infinite edge. Let $X$ be an attracting point of $\partial \psi$. Recall from corollary \ref{cor:negativeindex} that we must have
		$X = v*V$ for some word $v\in F_N$ and point $V\in \Sm\cup \Sp$. We suppose $V\in \Sp$; the other case is proved in a similar way.
		Considering $\psi\in \Phi^k$ and $\ind(\Phi^k)=\ind(\Phi_\infty)$, if $(U, V)\in \Sv$, then $v(U, V)$ is fixed by $\partial^2 \psi$.
		If $v(U, V)$ is a point of $\Sv$, we simply define $(U', V') = v(U, V)$ and $\psi_1 = \psi$. If not,
		let $v_{(1)} = v_h\dots v_p$ be the longest suffix of $v$ such that $v_p^{-1}\dots v_h^{-1}$ is a prefix of either $U$ or $V$.
		Define $v_{(0)} = v_0\dots v_{h-1}$ and $(U', V') = v_h\dots v_p(U, V)$.
		Observe $(U', V')$ is a point of $\Sv$ and we have $v_{h-1}\ne U_0'$ and $v_{h-1}\ne V_0'$.
		The automorphism $\psi_1 = i_{v_{(0)}}\circ \psi\circ i_{v_{(0)}^{-1}}$ is such that $\partial^2 \psi_1$ has
		a fixed point $(U', V')$ in $\Sv$ and $\partial \psi_1$ has a fixed point $v_{h-1}^{-1}\dots v_0^{-1}Y$ (for any $Y$
		in $H_\Omega$). Applying proposition \ref{prop:case1} tells us $\psi_1$ fixes a singularity $\Omega_1$
		which is in $\mathcal{G}_\Omega$ by corollary \ref{cor:concom}.
		
		Now, if there is a finite edge $(\Omega_1, \Omega_2, u)$ where $u$ is a prefix of $V'$ (note that by definition of
		finite edges, if $u_0 = V_0'$, then $u$ is a prefix of $V'$), then we define $(U'', V'') = u^{-1}(U', V')$
		and $\psi_2 = i_{u}\circ \psi_1\circ i_{u^{-1}}$.
		Iterate this process for as long as it is possible; this can only happen a finite number of times since there is a finite number of singularities
		(the index of an automorphism is finite) and since for any cycle $c$, the word $\gamma(c)$ can not be pure positive (from proposition \ref{prop:notpure}).
		We end up with an automorphism $\psi_n$ that fixes a singularity $\Omega_n$ of $\mathcal{G}_\Omega$ and such that $\partial \psi_n$ fixes
		a point $V^{(n)}\in \Sp$, with $V^{(n)} = v'V$ for some $v'\in F_N$. In addition, for any finite edge $(\Omega_n, \Omega_{n+1}, u')$, the word
		$u'$ is not a prefix of $V'$ and we conclude that $(\Omega_n, V^{(n)})$ is an infinite edge of $\mathcal{G}_\Omega$.

		The fact that an infinite edge represents an equivalence class of attracting points is clear from lemma \ref{lem:insidethm}.
		Finally, we deduce from proposition $(1)$ of this theorem, and from the definition of finite edges that
		an equivalence class can not be represented more than once, and we conclude.
		\end{proof}

		\subsubsection{Redefining the index}

		Constructing the singularity graph effectively gives us all the information needed to compute the FO-index of
		an automorphism. It also greatly details the structure of the fixed subgroups and the equivalence classes
		of attracting points. The singularity graph may actually be considered a little heavy to manipulate in case we're only interested
		in the FO-index. We show here that we can determine the FO-index with the singularities alone, regardless of
		the structure of the singularity graph.

		For two points $V, V'$ (resp $U, U'$) of $\Sp$ (resp. $\Sm$), we denote $V\sim V'$ (resp. $U\sim U'$)
		if $V_0 = V_0'$ (resp. $U_0=U_0'$).

		\begin{thm}
		The FO-index of $\varphi$ can be obtained by using the formula
		$$
			\ind(\Phi_\infty) = \frac{1}{2} \sum\limits_{\Omega} (\# (H_\Omega/\sim) - 2)
		$$
		where the sum is taken over all singularities.
		\end{thm}
		\begin{proof}
		Choose an integer $k$ such that $\ind(\Phi_\infty) = \ind(\Phi^k)$.
		Let $\mathcal{G}$ be the singularity graph, let $\Omega_0$ be a singularity fixed by $\psi_0\in \Phi^k$
		and let $\mathcal{G}_{\Omega_0}$ be the connected component of $\mathcal{G}$
		containing $\Omega_0$. Suppose $\mathcal{G}_{\Omega_0}$ contains $n$ nodes and $m$ finite edges, and observe that
		$rk(\pi_1(\mathcal{G}_{\Omega_0}))$ is simply given by $m-(n-1)$.

		We say that a finite edge of $\mathcal{G}$ is adjacent to a node $\Omega$ if there exist $u$ and $\Omega'$
		such that either $(\Omega, \Omega', u)$ or $(\Omega', \Omega, u)$ is a finite edge of $\mathcal{G}$.
		For any singularity $\Omega$ of $\mathcal{G}_{\Omega_0}$, it is easy to check that we have
		$$
			\#(H_\Omega/\sim) = \#(infinite~edges~anchored~at~\Omega) + \#(finite~edges~adjacent~to~\Omega).
		$$
		Obviously, a finite edge is adjacent to two nodes. Hence, taking the sum over all singularities of $\mathcal{G}_{\Omega_0}$
		and using theorem \ref{thm:graphindex} yields
		$$
			a(\psi_0) + 2~rk(Fix(\psi_0)) + 2(n-1) = \sum\limits_{\Omega\in \mathcal{G}_{\Omega_0}} \#(H_\Omega/\sim)
		$$
		(recall that $a(\psi_0)$ is the number of equivalence classes of attracting points of $\partial \psi_0$ (see section \ref{subsec:index}))
		and we conclude.
		\end{proof}

		We end this section by giving another formula for the FO-index. While slightly more complicated to write, it will allow
		for a more direct use of the informations given by the algorithm to come.

		First, observe that if $\Omega_\varphi$ is the singularity fixed by $\varphi^k$ for some interger $k\ge 1$, then
		for any two distinct points $V$ and $V'$ of $H_{\Omega_\varphi}$, we have $V_0\ne V_0'$.
		Indeed, any point $V$ of $\Sm\cup \Sp$ fixed by $\varphi^k$ verifies $V = \lim\limits_{n\to +\infty}\varphi^{kn}(V_0)$.

		Also, recall from proposition \ref{prop:commoncoord} that if $\psi=i_w\circ \varphi^k$ with $w\ne \epsilon$ fixes a singularity $\Omega$,
		then all the points of $\Omega$ have a common first or second coordinate.
		For two points $W=(U, V)$ and $W'=(U', V')$ of $\Sv$, we denote $W\approx W'$ if $U_0=U_0'$ and $V_0=V_0'$.

		\begin{thm}\label{thm:therealcount}
		Suppose no singularity is fixed by a power of $\varphi$. Then the FO-index of $\varphi$ is defined by
		$$
			\ind(\Phi_\infty) = \frac{1}{2} (\sum\limits_{\Omega} (\# (\Omega/\approx) - 1))
		$$
		where the sum is taken over all singularities.

		Suppose $\Omega_\varphi$ is the singularity fixed by a power of $\varphi$.
		The FO-index of $\varphi$ is then defined by
		$$
			\ind(\Phi_\infty) = \frac{1}{2} ((\# (H_{\Omega_\varphi}) - 2) + \sum\limits_{\Omega\ne \Omega_\varphi} (\# (\Omega/\approx) - 1))
		$$
		where the sum is taken over all singularities that are not $\Omega_\varphi$.
		\end{thm}

		This last formula is indeed quite powerful, as it allows us to obtain the FO-index while disregarding
		completely the possible fixed subgroups and the isogredience classes.

\section{Identifying singularities}\label{sec:identsing}

We now focus on the study of singularities. In section \ref{subsec:identifying}, we explain
how prefix-suffix developments can be used to identify points of a common singularity.
These results will provide a ground for section \ref{subsec:algorithm}, in which we give an algorithm
able to determine all the singularities in a finite number of steps.

	\subsection{Prefix-suffix developments of fixed points}\label{subsec:identifying}

	We assume again that $\varphi$ is $A_N$-positive primitive automorphism; its attracting subshift is denoted $\Sv$ and we define
	$\Phi_\infty = \{i_w\circ \varphi^k;~w\in F_N,~k\in \mathds{N}^*\}$. We also define the prefix-suffix representation map $\rho_{\varphi}$
	as in section \ref{subsec:psr}. We explained the action of the shift map $S$ on prefix-suffix developments, and we saw that we can not
	deduce the development $\rho_{\varphi}(S(W))$ (resp. $\rho_{\varphi}(S^{-1}(W)$) if the sequence of suffix (resp. prefix) of $\rho_{\varphi}(W)$ is constant equal to $\epsilon$.
	Lemma \ref{lem:nodeal} will allow us to avoid having to deal with these situations. As before, $\Svp\subset \Sv$ is the set of periodic points of $\partial^2 \varphi$.

	\begin{lem}\label{lem:nodeal}
	Let $(U, V)$ be a point of $\Sv$ fixed by $\partial^2 \psi$ for some automorphism $\psi=i_w\circ \varphi^k$ of $\Phi_\infty$ with $w\ne \epsilon$.
	If $w$ is pure positive (resp. negative), then $(U, V)$ is not a point of $\bigcup\limits_{n\in \mathds{N}} S^n(\Svp)$ (resp. $\bigcup\limits_{n\in \mathds{N}} S^{-n}(\Svp)$).
	\end{lem}
	\begin{proof}
	Let $(U', V')$ be a fixed point of $\partial^2 \varphi^k$ and suppose $(U, V)$ is on the $S$-orbit of $(U', V')$:
	we have $(U, V) = v^{-1}(U', V')$ for some prefix $v$ of $V'$ or $U'$. The word $v$ needs to be non empty if we want
	$w$ to be non empty. We get from proposition \ref{prop:isomove} that the point $(U, V)$ is fixed by
	$\partial^2 (i_v\circ \varphi^k\circ i_{v^{-1}}) = \partial^2 (i_{\varphi^k(v^{-1})v}\circ \varphi^k)$. Define $w = \varphi^k(v^{-1})v$.
	Since $v$ is a prefix of $V'$ or $U'$ and $(U', V')$ is fixed by $\partial^2 \varphi^k$, then $v$ must be a strict (from proposition \ref{prop:notpure})
	prefix of $\varphi^k(v)$. Hence, if $v$ is a prefix of $V'$ (resp. $U'$), then
	$w$ is pure negative (resp. positive).
	\end{proof}

	The following theorem is the base of the current study and justifies the use of prefix-suffix developments as a tool
	to find singularities.

	\begin{thm}\label{thm:fixedisperiodic}
	For any $A_N$-positive primitive automorphism $\varphi$,
	if $W\in \Sv$ is a fixed point of $\partial^2\psi$ with $\psi = i_w\circ\varphi^k\in \Phi_\infty$, then $\rho_{\varphi}(W)$ is preperiodic (eventually periodic) with a period of length $\le k$.
	\end{thm}
	\begin{proof}
	If $\rho_{\varphi}(W) = (p_i, a_i, s_i)_{i\ge 0}$, then $\rho_{\varphi}(\partial^2\varphi^k(W)) = (\epsilon, a_i', s_i')_{0\le i < k}(p_{i-k}, a_{i-k}, s_{i-k})_{i\ge k}$
	(see section \ref{subsec:psr}). We start with the case $w=\epsilon$. We have
	\begin{center}
		$\rho_{\varphi}(W) = (p_i, a_i, s_i)_{i\ge 0} = \rho_{\varphi}(\partial^2\varphi^k(W)) = (\epsilon, a_i', s_i')_{0\le i < k}(p_{i-k}, a_{i-k}, s_{i-k})_{i\ge k}$.
	\end{center}

	We now suppose $w$ is pure positive. The point $W$ is fixed by $S^{|w|}\circ\partial^2\varphi^k$ (proposition \ref{prop:freetocomb})
	and from lemma \ref{lem:nodeal}, we can deduce $\rho_{\varphi}(S^{-|w|}(W))$ from $\rho_{\varphi}(W)$.
	Applying $S^{-|w|}$ to the development of $W$ will only modify the $i_0$ first terms of the development, and we have
	\begin{center}
		\begin{tabular}{ccl}
			$\rho_{\varphi}(S^{-|w|}(W))$ & $=$ & $(q_i, b_i, r_i)_{0\le i < i_0}(p_i, a_i, s_i)_{i\ge i_0}$\\
			& $=$ & $\rho_{\varphi}(\partial^2\varphi^k(W))$\\
			& $=$ & $(\epsilon, a_i', s_i')_{0\le i < k}(p_{i-k}, a_{i-k}, s_{i-k})_{i\ge k}$.
		\end{tabular}
	\end{center}

	If $w$ is pure negative, then $W$ is fixed by $S^{-|w|}\circ\partial^2\varphi^k$ (proposition \ref{prop:freetocomb}) and
	the same arguments will apply for $\rho_{\varphi}(S^{|w|}(W))$.
	\end{proof}

	We are going to show that the converse also holds. 
	It will be easier for us to work with eventually constant development, rather than
	preperiodic ones. This can be achieved by considering powers of $\varphi$:
	recall that for all $k\in\mathds{N}^*$, $\varphi^k$ is also an $A_N$-positive primitive automorphism,
	and its attracting subshift is $\Sigma_{\varphi^k} = \Sv$.
	We denote by $\rho_{\varphi^k}$ ($k\ge 1$) the prefix-suffix development map associated to $\varphi^k$. 
	\begin{prop}
	Let $W$ be a point of $\Sv$. If $\rho_{\varphi}(W) = (p_i, a_i, s_i)_{i\ge 0}$
	then $\rho_{\varphi^k}(W) = (q_j, b_j, r_j)_{j\ge 0}$ with, for all $j\ge 0$,
	\begin{itemize}
		\item $q_j = \varphi^{k-1}(p_{jk+(k-1)})\varphi^{k-2}(p_{jk+(k-2)})\dots \varphi(p_{jk+1})p_{jk}$,
		\item $b_j = a_{jk}$,
		\item $r_j = s_{jk}\varphi(s_{jk+1})\dots \varphi^{k-2}(s_{jk+(k-2)})\varphi^{k-1}(s_{jk+(k-1)})$.
	\end{itemize}
	Consequently, if $\rho_{\varphi}(W)$ is preperiodic with a period of length $k$, then $\rho_{\varphi^k}(W)$ is eventually constant.
	\end{prop}

	We now only work with eventually constant development. An eventually constant
	development will be denoted $(p_i, a_i, s_i)_{0\le i < n}(p, a, s)*$, thereby
	signifying that the triplet $(p, a, s)$ is repeated indefinitely.
	We start with points with constant developments.
	\begin{lem}\label{lem:constantdev}
	Let $W=(U, V)$ be a point of $\Sv$ with $\rho_\varphi(W) = (p, a, s)*$.
	The point $W$ is fixed by $\partial^2 (i_p\circ \varphi)^h$
	for some integer $h\ge 1$. Moreover, if $p$ and $s$ are both non empty, then $h=1$.
	\end{lem}
	\begin{proof}
	Recall (section \ref{subsec:psr}) that if $p\ne \epsilon$, then
	$U = \lim\limits_{n\to +\infty} p^{-1}\varphi(p^{-1})\dots \varphi^n(p^{-1})$, and
	if $s\ne \epsilon$, then $V = \lim\limits_{n\to +\infty} as\varphi(s)\dots \varphi^n(s)$.
	Also recall we have $\varphi(a) = pas$.

	First observe that if $p = \epsilon$, then $V = \lim\limits_{n\to +\infty} \varphi^n(a)$. We conclude
	with proposition \ref{prop:uni-bi} that $W$ is fixed by $\partial^2 \varphi^h$ for some $h\ge 1$.

	If $p\ne \epsilon$, it is easy to see that $U$ is fixed by $\partial i_p\circ \varphi$.
	If $s = \epsilon$, then we use again proposition \ref{prop:uni-bi}. If we also have $s\ne \epsilon$,
	then we get from the fact that $\varphi(a) = pas$ that $W$ is fixed by $\partial^2 (i_p\circ \varphi)$.
	\end{proof}

	In section \ref{subsec:psr}, we have explained the action of the shift map on prefix-suffix developments.
	We can deduce from this study that if $W$ is a point of $\Sv$ with a preperiodic development,
	then there is an integer such that $\rho_\varphi(S^j(W))$ is constant. We obtain a converse to theorem
	\ref{thm:fixedisperiodic}.
	\begin{prop}\label{prop:periodicisfixed}
	Let $W$ be a point of $\Sv$ with an eventually constant development,
	then $W$ is fixed by $\partial^2 \psi$ for some $\psi\in \Phi_\infty$.
	\end{prop}

	Recall that the aim is to be able to find singularities; this goes through identifying
	points fixed by a common homeomorphism. We now know we only need to look among points with preperiodic (or eventually constant)
	prefix-suffix developments. To advance further, we are going to elaborate on proposition \ref{prop:periodicisfixed}.
	Namely, for a point $W$ fixed by $\partial^2 (i_w\circ \varphi)^h$,
	we are going to deduce $w$ from the prefix-suffix development of $W$. We will use the maps defined below.

	Let $\gm$ and $\gp$ be the $F_N$ to $F_N$ maps defined, for every $u=u_0u_1\dots u_p$ in $F_N$, by
	\begin{itemize}
		\item $\gm(u) = \varphi(u_p)u_0u_1\dots u_{p-1}$,
		\item $\gp(u) = u_1\dots u_{p-1}u_p\varphi(u_0)$.
	\end{itemize}
	\begin{thm}\label{thm:findingw}
	Let $\varphi$ be an $A_N$-positive primitive automorphism. Let $W$ be a point of $\Sv$ such that
	$\rho_\varphi(W) = (p_i, a_i, s_i)_{0\le i < n}(p, a, s)*$ and $\rho_{\varphi}(S^j(W)) = (p, a, s)*$ for some integer $j$.
	Then $W$ is a fixed point of $\partial^2 (i_w\circ \varphi)^h$ for some integer $h\ge 1$ and
	\begin{itemize}
		\item $(1)$ if $j\ge 0$, then $w = \gm^j(p)$,
		\item $(2)$ if $j < 0$, then $w^{-1} = \gp^{|j|-1}(s)$.
	\end{itemize}
	\end{thm}
	\begin{proof}
	$(1)$ The case $j=0$ is handled in lemma \ref{lem:constantdev}. We suppose $j>0$. Observe that $p$ can not be empty;
	if it were, the constant part of $\rho_\varphi(W)$ would not be $(p, a, s)$.
	Define $S^j(W) = (U, V)$ and let $u$ be the prefix of length $j$ of $U$; this gives $W = u^{-1}(U, V)$.
	We get from lemma \ref{lem:constantdev}
	and proposition \ref{prop:isomove} that there is an integer $h\ge 1$ such that $W$ is fixed by
	$\partial^2 (i_u\circ (i_p\circ \varphi)\circ i_{u^{-1}})^h = \partial^2 (i_w\circ \varphi)^h$ with
	$w = \varphi(u^{-1})pu$. Observe that $u = p^{-1}\dots \varphi^{k-1}(p^{-1})r^{-1}$ for some integer
	$k$ and some strict suffix $r$ of $\varphi^k(p) = qr$. We deduce $w = \varphi(r)q$.
	The word $p$ is obviously pure positive, and we deduce $j = |r|+\sum\limits_{0\le i\le k-1} |\varphi^i(p)|$.
	Note that $\gm^{|\varphi^i(p)|}(\varphi^i(p)) = \varphi^{i+1}(p)$ for any integer $i\ge 0$ and conclude
	$w = \gm^j(p)$.

	$(2)$ Suppose $j<0$, and observe again that $s$ can not be empty.
	Define $S^j(W) = (U, V)$ and let $v$ be the prefix of length $|j|$ of $V$, giving $W = v^{-1}(U, V)$. Similarly, there is an integer
	$h\ge 1$ such that $W$ is fixed by $\partial^2 (i_w\circ \varphi)^h$ with $w = \varphi(v^{-1})pv$. Observe
	that $v = as\dots \varphi^{k-1}(s)q$ for some integer $k$ and some strict prefix $q$ of $\varphi^k(s) = qr$.
	Recall $\varphi(a)=pas$ and deduce $w = \varphi(q^{-1})r^{-1}$. Again, $s$ is pure positive, and we
	obtain $|j| = |a|+|q| + \sum\limits_{0\le i\le k-1} |\varphi^i(s)|$. Since $\gp^{|\varphi^i(s)|}(\varphi^i(s)) = \varphi^{i+1}(s)$
	for any integer $i\ge 0$, we can conclude $w^{-1} = \gp^{|j|-1}(s)$.
	\end{proof}

	We are now able to state the main result of this section. The theorem below will be used
	as a base to identify singularities.

	\begin{thm}\label{thm:mainresult}
	Let $\varphi$ be an $A_N$-positive primitive automorphism.
	Let $W$ and $W'$ be two distinct points of $\Sv$ with $\rho_\varphi(W) = (p, a, s)*$
	and $\rho_\varphi(W') = (q, b, r)*$.
	\begin{itemize}
		\item $(1)$ If for any $i,j\in \mathds{N}$, we have $\gm^i(p)\ne \gm^j(q)$ (resp. $\gp^i(s)\ne \gp^j(r)$),
		then for any $i,j\in \mathds{N}$, the points $S^{-i}(W)$ and $S^{-j}(W')$ (resp. $S^{i+1}(W)$ and $S^{j+1}(W')$)
		do not belong to a common singularity.
		\item $(2)$ If $i$ and $j$ are the smallest integers such that $\gm^i(p) = \gm^j(q)$ (resp. $\gp^{i}(s) = \gp^{j}(r)$),
		then $S^{-i}(W)$ and $S^{-j}(W')$ (resp. $S^{i+1}(W)$ and $S^{j+1}(W')$) belong to the same singularity $\Omega$.
		Moreover, the singularity $\Omega$ is fixed by $(i_w\circ \varphi)^h$
		for some integer $h\ge 1$ and $w = \gm^i(p)$ (resp. $w^{-1} = \gp^{i}(s)$).
	\end{itemize}
	\end{thm}
	\begin{proof}
	Proposition $(1)$ is a direct consequence of theorem \ref{thm:findingw}.

	For proposition $(2)$, define $S^{-i}(W) = (U, V)$ and $S^{-j}(W') = (U', V')$
	(resp. $S^{i+1}(W) = (U, V)$ and $S^{j+1}(W') = (U', V')$) if $\gm^i(p) = \gm^j(q)$ (resp. $\gp^{i}(s) = \gp^{j}(r)$).
	Observe that $i$ and $j$ would not be the smallest integers for which
	$\gm^i(p) = \gm^j(q)$ (resp. $\gp^{i}(s) = \gp^{j}(r)$)
	if we had $V_0 = V_0'$ (resp. $U_0=U_0'$). Conclude with theorem \ref{thm:findingw}.
	\end{proof}

	Keeping the notation of the theorem above, note that it is possible for $S^{-i}(W)$ and $S^{-j}(W')$
	(resp. $S^{i+1}(W)$ and $S^{j+1}(W')$) to
	belong to the same singularity even if $i$ and $j$ are not the smallest integers for which
	$\gm^i(p) = \gm^j(q)$ (resp. $\gp^{i}(s) = \gp^{j}(r)$).
	This is typical of singularities containing two distinct points $(U, V)$ and $(U', V')$ with both $U_0 = U_0'$ and $V_0=V_0'$
	(see example \ref{subsec:ex1}).

	Also, it is important to recall that different points may have the same prefix-suffix development. This happens
	on the $S$-orbits of points of $\Svp$ (see section \ref{subsec:psr}). As theorem \ref{thm:mainresult} only take the prefix-suffix development
	into consideration, if using theorem \ref{thm:mainresult} tells us $S^m(W)$ and $S^n(W')$ belong to $\Omega$, then
	we also know that for any point $W_{(0)}$
	(resp. $W_{(0)}'$) such that $\rho_\varphi(W) = \rho_\varphi(W_{(0)})$ (resp. $\rho_\varphi(W') = \rho_\varphi(W_{(0)}')$),
	the point $S^m(W_{(0)})$ (resp. $S^n(W_{(0)}')$) belongs to $\Omega$.

	\subsection{An algorithm for finding singularities}\label{subsec:algorithm}

		\subsubsection{There is a only finite number of prefix-suffix developments to consider}\label{subsubsec:powerbound}

	We assume again $\varphi$ is an $A_N$-positive primitive automorphism, $\Sv$ is its attracting subshift
	and $\rho_\varphi$ is its prefix-suffix development map.
	We want to prove that singularities can be determined in a finite number of steps using the results of
	section \ref{subsec:identifying}. Namely, we are going to run pairs of points of $\Sv$ with constant
	prefix-suffix developments through theorem \ref{thm:mainresult}.
	Obviously, there are only a finite number of points $W$ such that $\rho_\varphi(W)$ is constant;
	it is given by
	$$\sum\limits_{a\in A_N} (number~of~occurences~of~a~in~\varphi(a)).$$
	However, in order to identify singularities, one should also work with points $W$ such that $\rho_{\varphi^k}(W)$ is constant
	for some integer $k\ge 1$. Bounding the number of such points then comes down to bounding the integer $k$.

	For a point $W$ of $\Sv$ with a preperiodic prefix-suffix development, the length of the minimal period of
	the periodic part of the prefix-suffix development of $W$ is simply called the $\boldsymbol{\rho}$\textbf{-power}
	of $W$.

	\begin{prop}\label{prop:rawpower}
	Let $W$ be a point of singularity $\Omega$. The $\rho$-power of $W$ is at most $4N-4$.
	\end{prop}
	\begin{proof}
	First assume that $W$ is in a singularity $\Omega$ and that $\rho_\varphi(W) = (p_i, a_i, s_i)_{i\ge 0}$
	with $p_i = \epsilon$ (resp. $s_i = \epsilon$) for any $i\ge i_0$. Observe that for any vertex $a$ of the prefix-suffix automaton,
	there is exactly one edge ending at $a$ whose label has an empty prefix (resp. suffix). We conclude
	the $\rho$-power of $W$ is at most $N$.

	Assume now that $W$ is in a singularity $\Omega$ with $\rho_\varphi(W) = (p_i, a_i, s_i)_{i\ge 0}$
	and that neither $(p_i)_{i\in \mathds{N}}$ nor $(s_i)_{i\in \mathds{N}}$ is eventually constant
	equal to $\epsilon$. Recall from section \ref{subsec:psr} that
	$\rho_\varphi(\partial^2 \varphi(W)) = (\epsilon, b, r)(p_{i-1}, a_{i-1}, s_{i-1})_{i\ge 1}$.
	Observe that $\Omega$ needs to be fixed by an automorphism $\psi = i_w\circ \varphi^k$ with $w\ne \epsilon$.
	Also note that there is an integer $j$ such that $S^j\circ \partial^2 \varphi(W)$ is in a singularity
	$\Omega'$ fixed by an automorphism $\psi' = i_{w'}\circ \varphi^{k'}$ with $w'\ne \epsilon$.
	All that is left to do is count the maximum possible number of distinct points
	of singularities that are not fixed by a power of $\varphi$.

	Recall the FO-index $\ind(\Phi_\infty)$ of $\varphi$ is at most $N-1$ (theorem \ref{thm:gjll})
	and theorem \ref{thm:therealcount} tells us
	$$
		\ind(\Phi_\infty) = \frac{1}{2} ((\# (H_{\Omega_\varphi}) - 2) + \sum\limits_{\Omega\ne \Omega_\varphi} (\# (\Omega/\approx) - 1))
	$$
	where $\Omega_\varphi$ is the singularity fixed by some power of $\varphi$ and the sum is taken over all singularities that are not $\Omega_\varphi$.
	It is now easy to see that the number of distinct points of singularities that are not fixed by a power of $\varphi$
	is maximal when
	\begin{itemize}
		\item no singularity is fixed by a power of $\varphi$,
		\item $\ind(\Phi_\infty) = N-1$
		\item there are $2N-2$ singularities, each containing two points.
	\end{itemize}
	In that case, we have actually shown that $W$ is fixed by an automorphism $\psi = i_w\circ \varphi^k$
	with $k\le 4N-4$, which is more than we wanted.
	\end{proof}

	This bound may very well not be sharp. Define the forward (resp. backward) singularities
	as the singularities such that all the points have a common first (resp. second) coordinate.
	It is likely, however unproven, that the forward (resp. backward) singularities may
	only be responsible for at most half of the maximum FO-index, dropping the bound of the proposition
	above to $2N-2$. One may check $2N-2$ is reached for any $N\ge 2$ by the automorphism
	\begin{center}
		\begin{tabular}{ccccl}
		$\phi$ & $:$ & $a_0$ & $\mapsto$ & $a_0a_1$\\
		&& $a_i$ & $\mapsto$ & $a_{i-1}$ for any $1\le i < N$
		\end{tabular}
	\end{center}

	We elaborate slightly on proposition \ref{prop:rawpower}. First, consider
	a point $W$ of a singularity $\Omega$. There is an automorphism $\psi = i_w\circ \varphi^k$
	such that $\partial^2 \psi(W)=W$, and we assume $k$ to be minimal. If $w\ne \epsilon$,
	then we obtain $k\le 4N-4$ from theorem \ref{thm:therealcount} (the argument is the same
	as the one given in the proof of proposition \ref{prop:rawpower} for points whose prefix-suffix developments
	do not have eventually empty prefixes or suffixes).
	However, as it is seen on the following example, the power $k$ may exceed $4N-4$ if $w$ is empty.
	\begin{center}
		\begin{tabular}{ccccl}
		$\phi$ & $:$ & $a$ & $\mapsto$ & $bf$\\
		&& $b$ & $\mapsto$ & $ca$\\
		&& $c$ & $\mapsto$ & $db$\\
		&& $d$ & $\mapsto$ & $ec$\\
		&& $e$ & $\mapsto$ & $ad$\\
		&& $f$ & $\mapsto$ & $e$\\
		\end{tabular}
	\end{center}
	Define $U = \lim\limits_{n\to +\infty} \varphi^{6n}(a^{-1})$ and
	$V = \lim\limits_{n\to +\infty} \varphi^{5n}(a)$. Since $aa$ is contained in $\varphi^2(b)$,
	then the point $(U, V)$ is in $\Sigma_\phi$. Moreover, $(U, V)$ is in a singularity and
	it is easy to check the smallest integer $k$ such that $\partial^2 \varphi^k(U, V) = (U, V)$ is $k=30$.

	Second, even if for any point $W$ of a singularity $\Omega$, there exists an automorphism $\psi = i_w\circ \varphi^k$
	with $k\le 4N-4$ such that $\partial^2 \psi(W)=W$, the singularity itself may not be fixed by an automorphism
	$i_v\circ \varphi^h$ with $h\le 4N-4$ (assuming $h$ minimal). Consider the following example on $14$ letters.
	\begin{center}
		\begin{tabular}{ccccl}
		$\phi$ & $:$ & $a$ & $\mapsto$ & $bcta$\\
		&& $b$ & $\mapsto$ & $a$\\
		&& $c$ & $\mapsto$ & $data$\\
		&& $d$ & $\mapsto$ & $e$\\
		&& $e$ & $\mapsto$ & $f$\\
		&& $f$ & $\mapsto$ & $g$\\
		&& $g$ & $\mapsto$ & $c$\\
		&& $t$ & $\mapsto$ & $uaca$\\
		&& $u$ & $\mapsto$ & $v$\\
		&& $v$ & $\mapsto$ & $w$\\
		&& $w$ & $\mapsto$ & $x$\\
		&& $x$ & $\mapsto$ & $y$\\
		&& $y$ & $\mapsto$ & $z$\\
		&& $z$ & $\mapsto$ & $t$\\
		\end{tabular}
	\end{center}
	Define $U = \lim\limits_{n\to +\infty} \varphi^{n}(a^{-1})$,
	\begin{itemize}
		\item $V^{(\alpha)} = \lim\limits_{n\to +\infty} \varphi^{2n}(\alpha)$ if $\alpha$ is $a$ or $b$,
		\item $V^{(\alpha)} = \lim\limits_{n\to +\infty} \varphi^{5n}(\alpha)$ if $\alpha$ is $c, d, e, f$ or $g$ and
		\item $V^{(\alpha)} = \lim\limits_{n\to +\infty} \varphi^{7n}(\alpha)$ if $\alpha$ is $t, u, v, w, x, y$ or $z$.
	\end{itemize}
	The points $(U, V^{(\alpha)}$) are fixed by
	\begin{itemize}
		\item $\partial^2 \varphi^2$ if $\alpha$ is $a$ or $b$,
		\item $\partial^2 \varphi^5$ if $\alpha$ is $c, d, e, f$ or $g$
		\item $\partial^2 \varphi^7$ if $\alpha$ is $t, u, v, w, x, y$ or $z$,
	\end{itemize}
	but the singularity $\Omega = \{(U, V^{(\alpha)}); \alpha\in \{a, b, c, d, e, f, g, t, u, v, w, x, y, z\}\}$ is fixed by $\varphi^{70}$.\\

	Thanks to proposition \ref{prop:rawpower}, we now know identifying the singularities can be done
	by running a finite number of pairs of points of $\Sv$ through theorem \ref{thm:mainresult}.
	We still need to verify that theorem \ref{thm:mainresult} effectively answers in a finite time.

		\subsubsection{Comparing two prefix-suffix developments takes a finite time}

	Let $W$ and $W'$ be two points of $\Sv$ such that $\rho_{\varphi^k}(W) = (p, a, s)*$ and
	$\rho_{\varphi^k}(W') = (q, b, r)*$. Suppose $i$ and $j$ are the smallest integers such that
	$w = \gmk^i(p) = \gmk^j(q)$. Then we can find a prefix $w_{(0)}$ of $w$
	such that $w_{(0)} = \varphi^k(ax)y_{(0)} = \varphi^k(by)$ where $a$ and $b$ are distinct elements of $A_N$,
	the words $x$ and $y$ are either empty or pure positive and $y_{(0)}$ is a strict suffix of $\varphi^k(y_p)$
	if $y_p$ is the last letter of $y$. Applying $\varphi^{-k}$ then gives $ax\varphi^{-k}(y_{(0)}) = by$.
	This can only happen if $\varphi^{-k}(y_{(0)}) = u*v$ with $u$ pure negative and such that $|u| = |ax|$
	and $v$ pure positive. Of course there is a similar reasoning to be made for $\gpk$.

	\begin{defn}\label{defn:gbound}
	For any $1\le k\le 4N-4$, define $\mathfrak{S}_k$ (resp. $\mathfrak{P}_k$) as the set of strict
	and non empty suffixes (resp. prefixes) $y$ (resp. $x$) of words $\varphi^k(a)$
	for all $a\in A_N$ that verify $\varphi^{-k}(y) = u_{(y)}*v_{(y)}$ (resp. $\varphi^{-k}(x) = v_{(x)}*u_{(x)}$)
	with $u_{(y)}$ (resp. $u_{(x)}$) empty or pure negative and $v_{(y)}$ (resp. $v_{(x)})$)
	non empty and pure positive.
	
	The $\boldsymbol{\gmk}$\textbf{-bound} (resp. $\boldsymbol{\gpk}$\textbf{-bound})
	is defined by $\max\limits_{y\in \mathfrak{S}_k} \{|u_{(y)}|\}$ (resp. $\max\limits_{x\in \mathfrak{P}_k} \{|u_{(x)}|\}$).
	\end{defn}

	We now prove it takes a finite time to determine whether two points with constant prefix-suffix developments
	can yield (using theorem \ref{thm:mainresult}) two points belonging to a common singularity.

	\begin{prop}\label{prop:wbound}
	Let $W$ and $W'$ be two points of $\Sv$ such that $\rho_{\varphi^k}(W) = (p, a, s)*$ and $\rho_{\varphi^k}(W') = (q, b, r)*$
	with $(p, a, s)\ne (q, b, r)$. There exist integers $i_0$ and $j_0$ such that,
	if for any $0\le m\le i_0$ and $0\le n\le j_0$, we have $\gmk^m(p)\ne \gmk^n(q)$, then
	for any $m, n\in \mathds{N}$, we have $\gmk^m(p)\ne \gmk^n(q)$. The same holds for $\gpk$.
	\end{prop}
	\begin{proof}
	In this proof, the $\gmk$-bound is simply denoted $g_k$.
	First, if $p=\epsilon$ (resp. $q=\epsilon$) and $q\ne \epsilon$ (resp. $p\ne \epsilon$), then
	we obviously have $\gmk^m(p)\ne \gmk^n(q)$ for any $m,n\in \mathds{N}$.

	Suppose now $p\ne \epsilon$ and $q\ne \epsilon$. Let $i$ (resp. $j$) be the smallest positive integer for which
	there exists a pure positive word $x$ (resp. $y$) with $|x| > g_k$ (resp. $|y| > g_k$) such that
	$\gmk^i(p) = \varphi^k(x)*x'$ (resp. $\gmk^j(q) = \varphi^k(y)*y'$) with $x'$ (resp. $y'$) empty or pure positive.
	Observe the existence of such integers is a consequence of the primitivity of $\varphi$.

	If $|\gmk^i(p)| > |\gmk^j(q)|$ (resp. $|\gmk^j(q)| > |\gmk^i(p)|$) then define $i_0=i$ (resp. $j_0=j$) and $j_0$ (resp. $i_0$)
	as the smallest integer such that $|\gmk^{j_0}(q)| > |\gmk^i(p)|$ (resp. $|\gmk^{i_0}(p)| > |\gmk^j(q)|$).
	If $|\gmk^i(p)| = |\gmk^j(q)|$, then define $i_0$ and $j_0$ to be the smallest integers such that $|\gmk^{i_0}(p)| > |\gmk^i(p)|$
	and $|\gmk^{j_0}(q)| > |\gmk^j(q)|$. Again, this is possible because $\varphi$ is primitive.\\

	We now assume for any $0\le m\le i_0$ and $0\le n\le j_0$, we have $\gmk^m(p)\ne \gmk^n(q)$.
	If $m > i_0$ and $n < j$ (resp. $m < i$ and $n > j_0$), then $|\gmk^m(p)| > |\gmk^n(q)|$ (resp. $|\gmk^n(q)| > |\gmk^m(p)|$).

	Suppose $m>i_0$ and $n\ge j$ and $\gmk^{m}(p) = \gmk^{n}(q)$.
	There are pure positive words $x_{(m)}$ (resp. $y_{(n)}$) with $|x_{(m)}| > g_k$ (resp. $|y_{(n)}| > g_k$)
	such that $\gmk^m(p) = \varphi^k(x_{(m)})*x_{(m)}'$ (resp. $\gmk^n(q) = \varphi^k(y_{(n)})*y_{(n)}'$)
	with $x_{(m)}'$ (resp. $y_{(n)}'$) empty or pure positive.
	If $|\varphi^k(x_{(m)})|\ge |\varphi^k(y_{(n)})|$ then there exists a possibly empty strict
	suffix $z$ of $\varphi^k(a)$ for some $a\in A_N$ such that $\varphi^k(y_{(n)})z = \varphi^k(x_{(m)}'')$
	where $x_{(m)}''$ is a prefix of $x_{(m)}$.
	Since $|y_{(n)}| > g_k$, this can only happen if $x_{(m)}$
	and $y_{(n)}$ agree on their first letter, which gives $\gmk^{m-1}(p) = \gmk^{n-1}(q)$. Of course,
	supposing $|\varphi^k(x_{(m)})| < |\varphi^k(y_{(n)})|$ yields a similar result, and we conclude
	iterating the process yields a contradiction.

	We proceed similarly for $m\ge i$ and $n>j_0$ and conclude.
	\end{proof}

	It should be noted the proof effectively constructs the integers $i_0$ and $j_0$.

		\subsubsection{The algorithm}

	Let $\varphi$ be an $A_N$-positive primitive automorphism and let $\Sv$ be its attracting subshift.
	The following algorithm determines all the singularities of $\Sv$ in a finite number of steps.

	\begin{enumerate}[(1)]
		\item Set $k=1$ and determine the $\gmk$-bound and $\gpk$-bound (definition \ref{defn:gbound}).
		\item Define $\mathcal{L}_k$ as the set of loops (cycles of length $1$) of the prefix-suffix automaton associated to $\varphi^k$.
		Note this does not actually require defining the whole automaton. Simply list all triplets $(p, a, s)$ of
		$F_N\times A_N\times F_N$ verifying $\varphi^k(a) = p*a*s$.
		\item For all pairs $((p, a, s), (q, b, r))$ of distinct elements of $\mathcal{L}_k$,
		\begin{enumerate}[(3.1)]
			\item search for the smallest integers (see proposition \ref{prop:wbound} for a condition of existence using the $\gmk$-bound)
			$i$ and $j$ such that $\gmk^i(p) = \gmk^j(q)$. If $i$ and $j$ exist, then
			\begin{enumerate}[(3.{1}.1)]
				\item define $w = \gmk^i(p) = \gmk^j(q)$,
				\item define a new singularity $\Omega$ containing all the points $W$ such that there exist $W_{(0)}$ (resp. $W_{(1)}$)
				with $W = S^{-i}(W_{(0)})$ and $\rho_{\varphi^k}(W_{(0)}) = (p, a, s)*$ (resp.
				$W = S^{-j}(W_{(1)})$ and $\rho_{\varphi^k}(W_{(1)}) = (q, b, r)*$) (see theorem \ref{thm:mainresult}).
				\item Associate the label $\psi = i_w\circ \varphi^k$ to the singularity $\Omega$. Observe that $\psi$
				does not necessarily fix the singularity (section \ref{subsubsec:powerbound}), but one of its power does.
				\item Search (among already known singularities) for a singularity $\Omega_0$ labeled $\psi_0$ with $\psi_0^h=\psi^{h'}$
				for some integers $h, h'$. If $\Omega_0$ exists, then define $\Omega' = \Omega\cup \Omega_0$, associate
				the label $\psi_0$ if $h\ge h'$ and $\psi$ if $h<h'$ (we keep the smallest automorphism as label)
				to $\Omega'$ and disregard singularities $\Omega$ and $\Omega_0$.
				\item Calculate the new current index using theorem \ref{thm:therealcount} and exit if it reaches $N-1$ (theorem \ref{thm:gjll}).
			\end{enumerate}
			\item Search for the smallest integers (see proposition \ref{prop:wbound} for a condition of existence using the $\gpk$-bound)
			$i$ and $j$ such that $\gpk^i(s) = \gpk^j(r)$. If $i$ and $j$ exist, then\label{test}
			\begin{enumerate}[(3.2.1)]
				\item define $w^{-1} = \gmk^i(p) = \gmk^j(q)$,
				\item define a new singularity $\Omega$ containing all the points $W$ such that there exist
				$W = S^{i+1}(W_{(0)})$ and $\rho_{\varphi^k}(W_{(0)}) = (p, a, s)*$ (resp.
				$W = S^{j+1}(W_{(1)})$ and $\rho_{\varphi^k}(W_{(1)}) = (q, b, r)*$) (see theorem \ref{thm:mainresult}).
				\item Reproduce steps (3.1.3), (3.1.4) and (3.1.5).
			\end{enumerate}
		\end{enumerate}
		\item If $k=4N-4$ then calculate the index using theorem \ref{thm:therealcount} and exit.
		If $k<4N-4$ then set $k=k+1$, determine the new $\gmk$-bound and $\gpk$-bound and go back to step (2).
	\end{enumerate}

	Note that the singularity graph can be easily constructed thanks to the next proposition.
	\begin{prop}\label{prop:sameorbit}
	Let $W$ and $W'$ be two points of $\Sv$ with $\rho_{\varphi^k}(W)=(p_i, a_i, s_i)_{0\le i < i_0}(p, a, s)*$ and
	$\rho_{\varphi^k}(W')=(q_i, b_i, r_i)_{0\le i < i_0}(p, a, s)*$.
	\begin{itemize}
		\item If $p\ne \epsilon$ and $s\ne \epsilon$ then $W$ and $W'$ are on the same $S$-orbit.
		\item If $p=\epsilon$ or $s=\epsilon$, then there is a point $W''$ such that
		$\rho_{\varphi^k}(W')=\rho_{\varphi^k}(W'')$ and $W$ and $W''$ are on the same $S$-orbit.
	\end{itemize}
	\end{prop}
	One may also observe that two points $W$ and $W'$ of $\Sv$ with
	$\rho_{\varphi^k}(W)=(p_i, a_i, s_i)_{0\le i < i_0}(p, a, s)*$ and
	$\rho_{\varphi^k}(W')=(q_i, b_i, r_i)_{0\le i < i_0}(q, b, r)*$ and $(p, a, s)\ne (q, b, r)$
	may still be on the same orbit if (and this is not sufficient) $p=\epsilon$ (resp. $s=\epsilon$)
	and $r=\epsilon$ (resp. $q=\epsilon$).

\section{Examples}\label{sec:examples}

	\subsection{First example}\label{subsec:ex1}

	Consider the $\{a, b, c\}$-positive primitive automorphism $\phi$ defined by
	\begin{center}
		\begin{tabular}{cccclcccccl}
		$\phi$ & $:$ & $a$ & $\mapsto$ & $ba$ & \hspace{6mm} & $\phi^{-1}$ & $:$ & $a$ & $\mapsto$ & $c^{-1}a$\\
		&& $b$ & $\mapsto$ & $babac$ & & & & $b$ & $\mapsto$ & $c$\\
		&& $c$ & $\mapsto$ & $b$ & & & & $c$ & $\mapsto$ & $a^{-1}a^{-1}b$
		\end{tabular}
	\end{center}
	The set of loops of the prefix-suffix automaton associated to $\phi$ is simply
	$$\mathcal{L}_1 = \{(b, a, \epsilon),~(\epsilon, b, abac),~(ba, b, ac)\}.$$
	One may check that no singularities are created, and we move on to $\phi^2$.
	\begin{center}
		\begin{tabular}{cccclcccccl}
		$\phi^2$ & $:$ & $a$ & $\mapsto$ & $babacba$ & \hspace{6mm} & $\phi^{-2}$ & $:$ & $a$ & $\mapsto$ & $b^{-1}aac^{-1}a$\\
		&& $b$ & $\mapsto$ & $babacbababacbab$ & & & & $b$ & $\mapsto$ & $a^{-1}a^{-1}b$\\
		&& $c$ & $\mapsto$ & $babac$ & & & & $c$ & $\mapsto$ & $a^{-1}ca^{-1}cc$
		\end{tabular}
	\end{center}
	The set $\mathcal{L}_2$ of loops of the prefix-suffix automaton associated to $\phi^2$ is quite large, and we only mention
	relevant triplets. Obviously, there is a singularity
	$$\Omega_0 = \{S^1(W)\in \Sigma_\phi;~\rho_{\varphi^2}(W)\in \{(babacb, a, \epsilon),~(babacbababacba, b, \epsilon),~(baba, c, \epsilon)\}\}.$$
	Considering $\phi^2(u)$ starts with the letter $b$ for any letter $u\in A_N$, we deduce $\Omega_0$ contains exactly $3$ points.
	Also observe the triplets $(babacbaba, b, acbab),~(babac, b, ababacbab)$ are in $\mathcal{L}_2$ and we have
	$$\gamma_{\varphi_-^2}(babac) = \varphi^2(c)baba = babacbaba$$
	which gives the singularity $\Omega_1 = \{W_{(1)}, W_{(1)}'\}$ with
	\begin{center}
		$\rho_{\varphi^2}(W_{(1)})=(babacbaba, b, acbab)*$~~ and~~ $\rho_{\varphi^2}(S^1(W_{(1)}'))=(babac, b, ababacbab)*$.
	\end{center}
	Finally, the triplet $(b, a, bacba)$ is also in $\mathcal{L}_2$ and we have
	$$\gamma_{\varphi_-^2}(b) = \gamma_{\varphi_-^2}(babacbaba)$$
	which gives a singularity $\Omega_2 = \{W_{(2)}, S^{-1}(W_{(1)}), S^{-1}(W_{(1)}')\}$ with
	\begin{center}
		$\rho_{\varphi^2}(S^1(W_{(2)})) = (b, a, bacba)*$.
	\end{center}

	At that point, $\#(H_{\Omega_0}) = 4$, $\#(\Omega_1/\approx) = 2$ and $\#(\Omega_2/\approx) = 2$ (because $S^{-1}(W_{(1)})\approx S^{-1}(W_{(1)}')$)
	and the FO-index of $\phi$ reaches the maximum possible value of $2$.
	We obtain the singularity graph described on figure \ref{fig:SG1}, which contains two connected components. The infinite edges labeled with an element of $\Sigma_\phi^-$
	(resp. $\Sigma_\phi^+$) are conventionally drawn going left (resp. right).

	\begin{figure}[h!]
	\begin{center}
		\scalebox{0.48}{\includegraphics{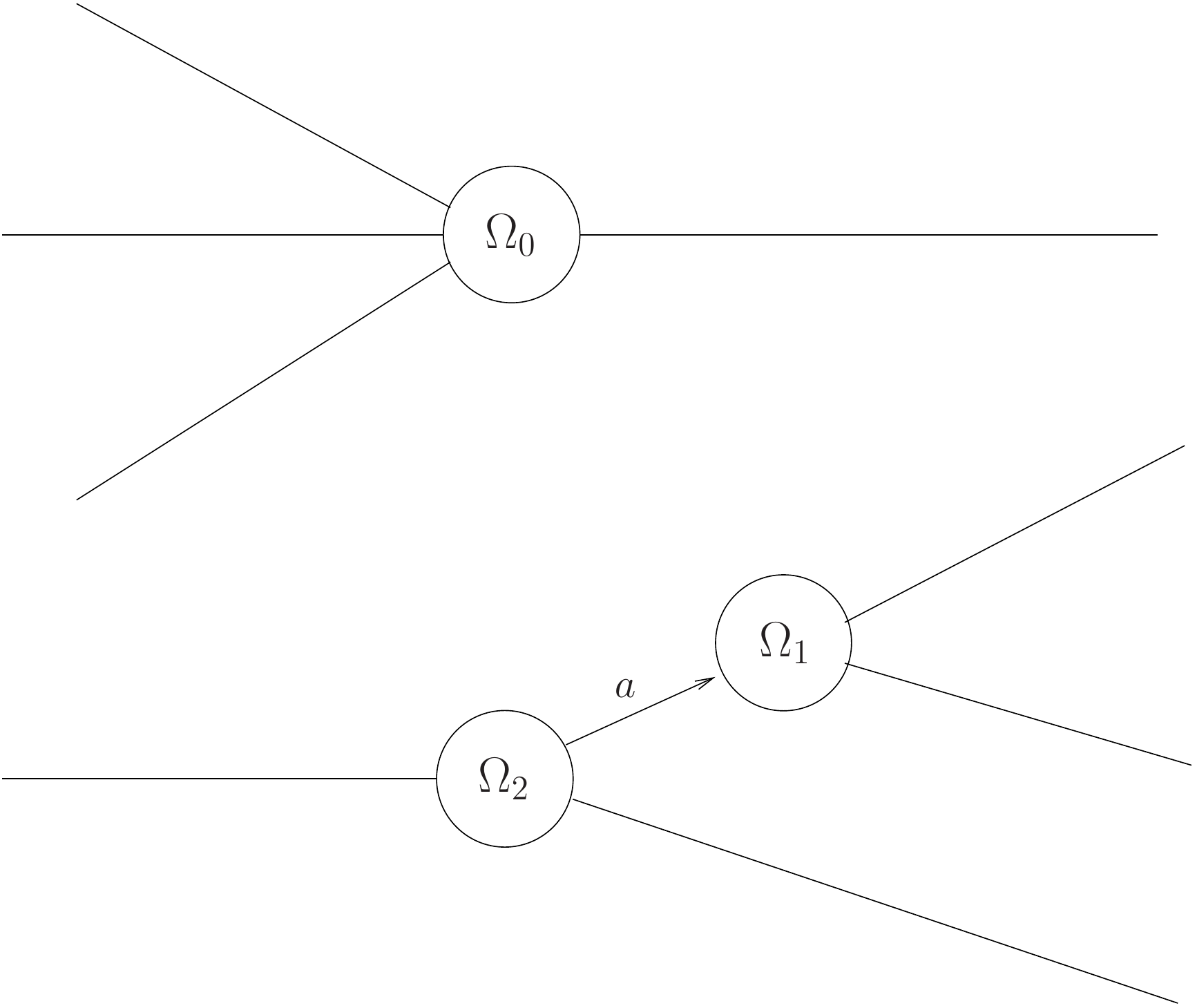}}
	\end{center}
	\vspace{-4mm}
	\caption{Singularity graph associated to $\phi$.}
	\label{fig:SG1}
	\end{figure}

	\subsection{Second example}

	The following example is a standard automorphism coming from a pseudo-Anosov on a surface of genus $2$
	with one boundary component. Alternatively, one can see it as the coding automorphism associated
	to an interval exchange transformation.
	Consider the $\{a, b, c, d\}$-positive primitive automorphism $\phi$ defined by
	\begin{center}
		\begin{tabular}{cccclcccccl}
		$\phi$ & $:$ & $a$ & $\mapsto$ & $abdacd$ & \hspace{6mm} & $\phi^{-1}$ & $:$ & $a$ & $\mapsto$ & $ad^{-1}b^{-1}ad^{-1}$\\
		&& $b$ & $\mapsto$ & $abdbd$ & & & & $b$ & $\mapsto$ & $da^{-1}bd^{-1}cd^{-1}ad^{-1}b^{-1}ad^{-1}$\\
		&& $c$ & $\mapsto$ & $accd$ & & & & $c$ & $\mapsto$ & $da^{-1}bda^{-1}cd^{-1}ad^{-1}b^{-1}ad^{-1}$\\
		&& $d$ & $\mapsto$ & $acd$ & & & & $d$ & $\mapsto$ & $da^{-1}bda^{-1}dc^{-1}d$
		\end{tabular}
	\end{center}
	Note that, since $\phi(u)$ starts with $a$ and ends with $d$ for any letter $u\in A_N$, we can deduce
	the prefix-suffix development map is a bijection. The set of loops of the prefix-suffix automaton associated to $\phi$ is
	$$\mathcal{L}_1 = \{(\epsilon, a, bdacd),~(abd, a, cd),~(a, b, dbd),~(abd, b, d),~(a, c, cd),~(ac, c, d),~(ac, d, \epsilon)\}.$$

	There are a lot of obvious pairs yielding singularities. We get:
	\begin{itemize}
		\item $\Omega_0 = \{W_{(0)}, W_{(0)}'\}$ with $\rho_\varphi(W_{(0)}) = (a, b, dbd)*$ and $\rho_\varphi(W_{(0)}') = (a, c, cd)*$,
		\item $\Omega_1 = \{W_{(1)}, W_{(1)}'\}$ with $\rho_\varphi(W_{(1)}) = (abd, a, cd)*$ and $\rho_\varphi(W_{(1)}') = (abd, b, d)*$,
		\item $\Omega_2 = \{W_{(2)}, W_{(2)}'\}$ with $\rho_\varphi(W_{(2)}) = (ac, c, d)*$ and $\rho_\varphi(W_{(2)}') = (ac, d, \epsilon)*$,
		\item $\Omega_3 = \{W_{(3)}, W_{(3)}'\}$ with $\rho_\varphi(S^{-1}(W_{(3)})) = (abd, a, cd)*$ and $\rho_\varphi(S^{-1}(W_{(3)}')) = (a, c, cd)*$,
		\item $\Omega_4 = \{W_{(4)}, W_{(4)}'\}$ with $\rho_\varphi(S^{-1}(W_{(4)})) = (abd, b, d)*$ and $\rho_\varphi(S^{-1}(W_{(4)}')) = (ac, c, d)*$.
	\end{itemize}
	Also, observe $\gamma_{\varphi_+}(dbd) = bdacd$, and deduce there is a singularity $\Omega_5 = \{W_{(5)}, W_{(5)}'\}$ with
	$\rho_\varphi(S^{-2}(W_{(5)})) = (a, b, dbd)*$ and $\rho_\varphi(S^{-1}(W_{(5)}')) = (\epsilon, a, bdacd)*$.
	Observe that these singularities give $\phi$ a maximum FO-index of $3$.
	We obtain the singularity graph of figure \ref{fig:SG2}, composed of a single connected component.

	\begin{figure}[h!]
	\begin{center}
		\scalebox{0.55}{\includegraphics{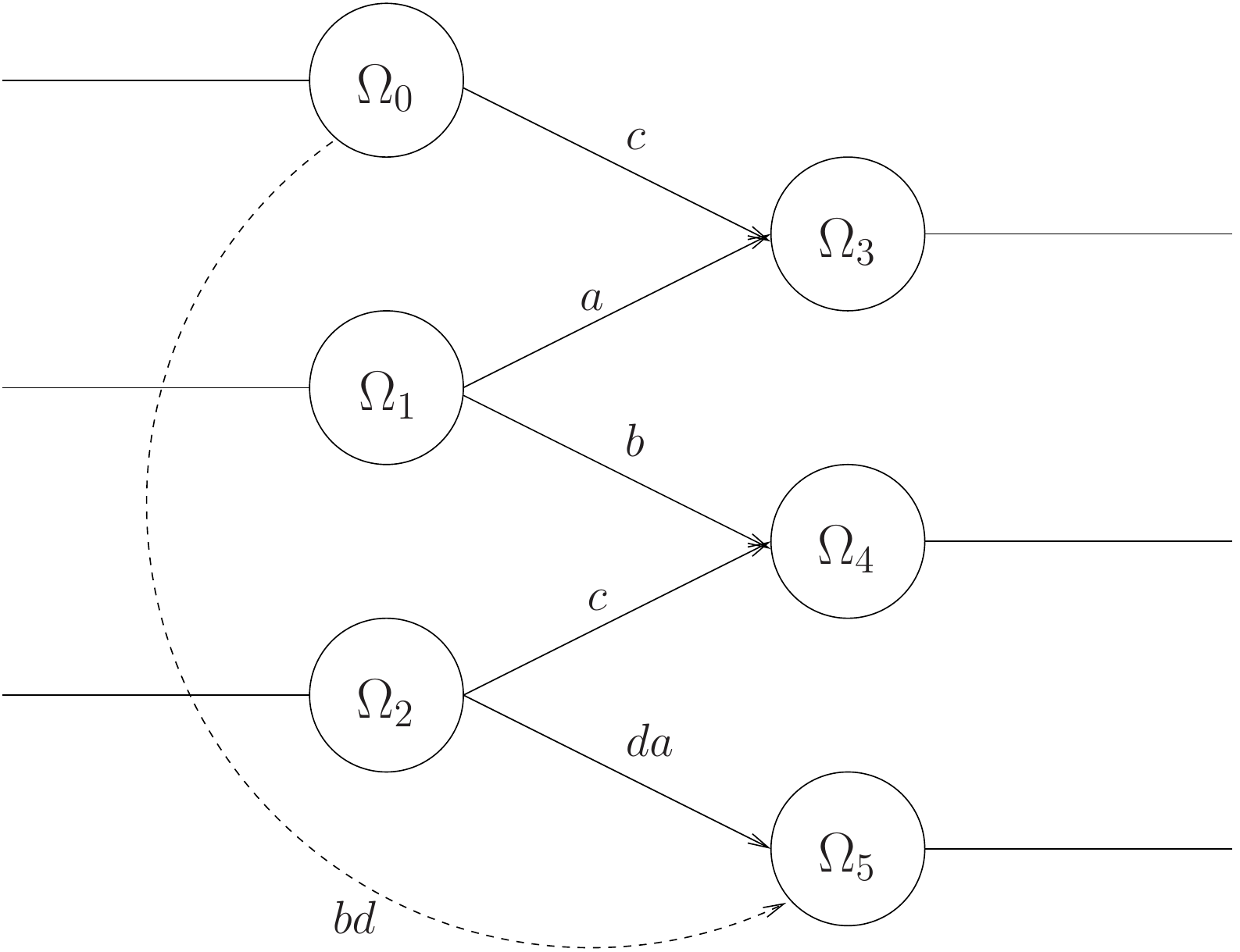}}
	\end{center}
	\vspace{-4mm}
	\caption{Singularity graph associated to $\phi$.}
	\label{fig:SG2}
	\end{figure}

	The graph contains a cycle which, as stated (theorem \ref{thm:insidethm}), tells us about the fixed subgroups of the automorphisms fixing the singularities.
	For example, $\Omega_0$ is fixed by $i_a\circ \phi$, and one can check that the cycle
	$$ca^{-1}bc^{-1}da(bd)^{-1}$$
	starting and ending at $\Omega_0$ in the singularity graph is effectively fixed by $i_a\circ \phi$.

\bibliography{bibli}{}

\begin{thebibliography}{10}

\bibitem{BFH}
Mladen Bestvina, Mark Feighn, and Michael Handel.
\newblock Laminations, {T}rees, and {I}rreducible {A}utomorphisms of {F}ree
  {G}roup.
\newblock {\em Geometric And {F}unctional {A}nalysis}, 7:215--244, 1997.

\bibitem{BH}
Mladen Bestvina and Michael Handel.
\newblock Train-tracks and automorphisms of free groups.
\newblock {\em Annals of {M}athematics}, 135:1--51, 1992.

\bibitem{BK}
Michael Boshernitzan and Isaac Kornfeld.
\newblock Interval translation mappings.
\newblock {\em Ergodic {T}heory {D}ynam. {S}ystems}, 15(5):821--832, 1995.

\bibitem{CanSie}
Vincent Canterini and Anne Siegel.
\newblock Automate des pr\'efixes-suffixes associ\'e \`a une substitution
  primitive.
\newblock {\em Journal de {T}h\'eorie des {N}ombres de {B}ordeaux},
  13:353--369, 2001.

\bibitem{CL}
Marshall~M. Cohen and Martin Lustig.
\newblock On the dynamics of the fixed subgroup of a free group automorphism.
\newblock {\em Inventiones {M}athematicae}, 89:613--638, 1989.

\bibitem{Coo}
Daryl Cooper.
\newblock Automorphisms of free groups have finitely generated fixed point
  sets.
\newblock {\em Journal of {A}lgebra}, 111:453--456, 1987.

\bibitem{CH}
Thierry Coulbois and Arnaud Hilion.
\newblock Rips induction: {I}ndex of the dual lamination of an
  $\mathds{R}$-tree.
\newblock 2010.
\newblock arXiv:1002.0972v2.

\bibitem{CHL09}
Thierry Coulbois, Arnaud Hilion, and Martin Lustig.
\newblock $\mathds{R}$-trees, dual laminations, and compact systems of partial
  isometries.
\newblock {\em Math. {P}roc. {C}ambridge {P}hil. {S}oc.}, 147:345--368, 2009.

\bibitem{GJLL}
Damien Gaboriau, Andre Jaeger, Gilbert Levitt, and Martin Lustig.
\newblock An index for counting fixed points of automorphisms of free groups.
\newblock {\em Duke {M}ath. {J}.}, 93:425--452, 1998.

\bibitem{GLL}
Damien Gaboriau, Gilbert Levitt, and Martin Lustig.
\newblock A dendrological proof of the {S}cott conjecture for automorphisms of
  free groups.
\newblock {\em Proceedings of the {E}dinburgh {M}athematical {S}ociety},
  41(2):325--332, 1998.

\bibitem{Jul}
Yann Jullian.
\newblock Construction du c{\oe}ur compact d'un arbre réel par substitution
  d'arbre.
\newblock {\em Annales de l'{I}nstitut {F}ourier}, 2011.
\newblock (to appear).

\bibitem{LL03}
Gilbert Levitt and Martin Lustig.
\newblock Irreducible automorphisms of ${F}_n$ have north-south dynamics on
  compactified outer-space.
\newblock {\em Journal of the {I}nst. of {M}ath. {J}ussieu}, 2:59--72, 2003.

\bibitem{Que}
M.~Queff\'elec.
\newblock {\em Substitution {D}ynamical {S}ystems-{S}pectral {A}nalysis}.
\newblock Lecture {N}otes in {M}athematics, 1294. Springer-Verlag, 1987.

\end{thebibliography}
\bibliographystyle{plain}

\end{document}